\documentclass[a4paper,twoside]{amsart}
\usepackage[utf8]{inputenc}
\usepackage[T1]{fontenc}
\usepackage[british]{babel}
\usepackage{amssymb,amsmath,amsthm,amscd}
\usepackage{mathrsfs}
\usepackage{enumerate}
\usepackage{verbatim}
\usepackage{lmodern}

\newcommand{\C}{\ensuremath{\mathbb{C}}}
\newcommand{\R}{\ensuremath{\mathbb{R}}}
\newcommand{\N}{\ensuremath{\mathbb{N}}}

\newcommand{\cG}{\ensuremath{\mathcal{G}}}

\newcommand{\cM}{\ensuremath{\mathcal{M}}}
\newcommand{\cO}{\ensuremath{\mathcal{O}}}

\newcommand{\acts}{\ensuremath{\curvearrowright}}
\renewcommand\epsilon\varepsilon

\theoremstyle{definition}
\newtheorem{thmA}{Theorem}

\newtheorem{corA}[thmA]{Corollary}
\newtheorem{thm}{Theorem}[section]
\newtheorem{dfn}[thm]{Definition}
\newtheorem{lem}[thm]{Lemma}
\newtheorem{prp}[thm]{Proposition}
\newtheorem{cor}[thm]{Corollary}

\newtheorem{rem}[thm]{Remark}
\newtheorem{qst}[thm]{Question}

\author[de Laat]{Tim de Laat}
\thanks{TdL is supported by the Deutsche Forschungsgemeinschaft under Germany's Excellence Strategy -- EXC 2044 -- 390685587, Mathematics M\"unster: Dynamics -- Geometry -- Structure and through SFB 878}
\address{Tim de Laat \newline Mathematisches Institut, Westf\"alische Wilhelms-Universit\"at M\"unster \newline Einsteinstrasse 62, 48149, M\"unster, Germany}
\email{tim.delaat@uni-muenster.de}

\author[de la Salle]{Mikael de la Salle}
\thanks{MdlS is supported by the CNRS and by the LABEX MILYON (ANR-10-LABX-0070) of Universit\'e de Lyon, within the program ``Investissements d'Avenir'' (ANR-11-IDEX-0007) operated by the French National Research Agency (ANR). His research is also supported by the ANR projects GAMME (ANR-14-CE25-0004) and AGIRA (ANR-16-CE40-0022).}
\address{Mikael de la Salle \newline UMPA, CNRS--ENS de Lyon \newline 69364 Lyon cedex 7, France}
\email{mikael.de.la.salle@ens-lyon.fr}

\title{Banach space actions and $L^2$-spectral gap}
  
\begin{document}

\begin{abstract}
  \.{Z}uk proved that if a finitely generated group admits a Cayley graph such that the Laplacian on the links of this Cayley graph has a spectral gap $> \frac{1}{2}$, then the group has property (T), or equivalently, every affine isometric action of the group on a Hilbert space has a fixed point. We prove that the same holds for affine isometric actions of the group on a uniformly curved Banach space (for example an $L^p$-space with $1 < p < \infty$ or an interpolation space between a Hilbert space and an arbitrary Banach space) as soon as the Laplacian on the links has a two-sided spectral gap $>1-\varepsilon$.
  
  This criterion applies to random groups in the triangular density model for densities $> \frac{1}{3}$. In this way, we are able to generalize recent results of Dru\c{t}u and Mackay to affine isometric actions of random groups on uniformly curved Banach spaces. Also, in the setting of actions on $L^p$-spaces, our results are quantitatively stronger, even in the case $p=2$. This naturally leads to new estimates on the conformal dimension of the boundary of random groups in the triangular model.

  Additionally, we obtain results on the eigenvalues of the $p$-Laplacian on graphs, and on the spectrum and degree distribution of Erd\H{o}s-R\'enyi graphs.
\end{abstract}

\maketitle

\section{Introduction and main results}

\subsection{Introduction}
Fixed point properties for group actions on metric spaces, e.g.~Banach spaces or non-positively curved spaces, are natural rigidity properties that contribute to the understanding of both groups and the spaces on which they act. When considering actions on Banach spaces, the natural actions to consider are affine isometric actions. Given a Banach space $X$, a topological group is said to have property (F$_X$) if every continuous affine isometric action of the group on $X$ has a fixed point. In this article, we deal with fixed point properties for countable discrete groups. In this setting, every affine isometric action is automatically continuous.

Property (F$_X$) was introduced by Bader, Furman, Gelander and Monod \cite{MR2316269} as a Banach space version of Serre's property (FH). A topological group has property (FH) if every continuous affine isometric action of the group on a Hilbert space has a fixed point. It is well known that a countable group has property (FH) if and only if it has property (T), which is a rigidity property for groups that was introduced by Kazhdan \cite{MR0209390}. A group has property (T) if its trivial representation is isolated in the unitary dual of the group equipped with the Fell topology. Both property (T) and property (FH) have lead to striking results in several areas of mathematics, e.g.~group theory, combinatorics, ergodic theory, dynamical systems, measure theory and operator algebras. We refer to \cite{MR2415834} for a detailed account of property (T) and property (FH).

Partly because of the aforementioned connections to different areas of mathematics, recent years have seen a growing interest in Banach space versions of both fixed point properties and property (T). Alongside property (F$_X$), as recalled above, Bader, Furman, Gelander and Monod also defined a Banach space version of property (T), which is called property (T$_X$) and is in general weaker than property (F$_X$) (see \cite[Theorem 1.3]{MR2316269}). Another notable Banach space strengthening of property (T) is strong property (T), which Lafforgue introduced in his work on the Baum-Connes conjecture \cite{MR2423763,MR2574023}. He essentially proved that if a group has strong property (T) relative to a Banach space $X \oplus \mathbb{C}$, then the group has property (F$_X$).

The most straightforward non-Hilbertian Banach spaces to consider are $L^p$-spaces, with $p \neq 2$. For $1 \leq p < \infty$, a countable group is said to have property (F$L^p$) if every affine isometric action of the group on an $L^p$-space has a fixed point. It is known that property (T) implies property (F$L^p$) for $p \in [1,2+\varepsilon)$, where $\varepsilon$ may depend on the group (see \cite[Theorem 1.3]{MR2316269} (and also \cite{drutukapovich}) for the case $p \in (1,2+\varepsilon)$ and \cite[Corollary D]{MR2929085} for $p=1$). In several cases, there are explicit lower bounds on $\varepsilon$ (see \cite{MR2978328,MR3606450,drutumackay}). On the other hand, there are groups with property (T) that are known to fail property (F$L^p$) for large $p$ \cite{pansu1995,MR1979183,MR2221161,MR2421319}, e.g.~cocompact lattices in $\mathrm{Sp}(n,1)$. However, (lattices in) connected simple higher-rank Lie groups and (lattices in) connected simple higher-rank algebraic groups over non-Archimedean local fields have property (F$L^p$) for all $p \in [1,\infty)$ (see \cite[Theorem B]{MR2316269} and \cite[Corollary D]{MR2929085}). Similar results have been established for universal lattices \cite{MR2794627}.

Bader, Furman, Gelander and Monod conjectured that (lattices in) connected simple higher-rank Lie groups and (lattices in) connected simple higher-rank algebraic groups over non-Archimedean local fields have property (F$_X$) for every superreflexive Banach space $X$ \cite[Conjecture 1.6]{MR2316269}. This conjecture has been proved in the non-Archimedean setting \cite{MR2423763,MR3190138}, and in the real and complex case, partial results have been obtained \cite{MR3474958,MR3407190,MR3533271}. Other results that show fixed point properties by means of an appropriate strengthening of property (T) were obtained by Oppenheim \cite{MR3606450}. His examples include certain groups acting on buildings and Kac-Moody-Steinberg groups.

Another effective way of establishing fixed point properties or property (T) for a group is by means of spectral conditions on the links of vertices of certain simplicial complexes on which the group acts. The idea of this method goes back to \cite{MR0320180} and was further developed in \cite{MR1651383,MR1408975,MR1995802,MR1465598} in order to provide criteria to establish property (T). Nowadays, the most well-known spectral criterion for property (T) may be the one due to \.{Z}uk \cite{MR1995802}, asserting that if $\Gamma$ is a finitely generated group with finite symmetric generating set $S$ (with $1 \notin S$) such that the link graph $L(S)$ associated with $S$ is connected and the smallest non-zero eigenvalue of the Laplacian on $L(S)$ is strictly larger than $\frac{1}{2}$, then $\Gamma$ has property (T).

In recent years, certain local criteria for fixed point properties for group actions on Banach spaces have been established, by Bourdon \cite{MR2978328} for actions on $L^p$-spaces, and by Nowak \cite{MR3323202} and by Oppenheim \cite{MR3278887} for actions on reflexive spaces. Oppenheim also explains that the assumption of reflexivity is not needed in his approach, and his proof is elementary.

The use of spectral criteria is particularly beneficial when considering random groups. The framework of random groups provides ways to consider finitely presented groups in which the relators are chosen at random according to some prescribed probability measure on the set of all possible words in the generating set. The theory of random groups goes back to \cite{MR1253544}, in which Gromov introduced what is now called the Gromov density model $\cG(n,l,d)$ (see also \cite{MR1978492}), in which the density $d$ is a parameter that controls the number of relators. It was proved by Gromov that for $d < \frac{1}{2}$, a random group in $\cG(n,l,d)$ is infinite and hyperbolic with overwhelming probability (w.o.p.), whereas for $d \geq \frac{1}{2}$, a group in $\cG(n,l,d)$ is trivial or $\mathbb{Z}_2$ w.o.p.~\cite{MR1253544} (see also \cite{MR2205306}). The study of property (T) for random groups was initiated by \.{Z}uk \cite{MR1995802}. By using his aforementioned criterion, he proved that for $d > \frac{1}{3}$, a random group in the triangular density model $\cM(m,d)$, which is an adaptation of the Gromov density model that is particularly suitable for the use of the spectral criterion, has property (T) w.o.p. The fact that for $d > \frac{1}{3}$, a group in the Gromov density model $\cG(n,l,d)$ has property (T) w.o.p., was proved in detail in \cite{MR3106728}.

Fixed point properties for actions of random groups on $L^p$-spaces were first considered by Nowak, by applying his spectral criterion mentioned above \cite[Section 6]{MR3323202}. Moreover, in a recent article, Dru\c{t}u and Mackay made substantial contributions to the understanding of property (F$L^p$) in the setting of random groups \cite{drutumackay}. The main part of their argument consists of establishing new bounds on the first positive eigenvalue of the $p$-Laplacian on random graphs. By applying Bourdon's criterion, they obtain fixed point properties of actions of random groups on $L^p$-spaces. In fact, Dru\c{t}u and Mackay do not restrict to actions on $L^p$-spaces, but consider the more general setting of actions on Banach spaces whose finite-dimensional subspaces are $\alpha$-isomorphic to a subspace of an $L^p$-space, with $\alpha< 2^{\frac{1}{2p}} < 1.19$. Additionally, their results lead to quantitative results on the conformal dimension of the boundary of random groups. We will elaborate more on their results below.

\subsection{Statement of the main results} \label{subsec:mainresults}
The aim of this article is two-fold. First we establish a criterion for groups that ensures that every affine isometric action of the group on a given uniformly curved Banach space has a fixed point. Uniform curvedness is a property introduced by Pisier (see Section \ref{subsection:unif_curved}). Examples of uniformly curved spaces are $L^p$-spaces with $1 < p < \infty$ and interpolation spaces between a Hilbert space and an arbitrary Banach space, i.e.~strictly $\theta$-Hilbertian spaces. Uniform curvedness is stable under renorming and finite representability.

Second, we apply our spectral criterion to random groups in the triangular density model, giving the first results on fixed point properties for actions of random groups on Banach spaces that are very different from $L^p$-spaces. Also, in the setting of $L^p$-spaces (even in the case $p=2$), our results strengthen the known results on fixed point properties of random groups in the triangular density model. We also establish new quantitative results on the conformal dimension of the boundary of random groups.

We only consider complex Banach spaces, but is is straightforward to formulate our results in the setting of real Banach spaces.

In what follows, if $\mathcal{G}=(V,\omega)$ is a connected finite (weighted) graph, we denote by $A_{\mathcal{G}}$ the Markov operator of the random walk on $\mathcal{G}$ (see Section \ref{sec:graphs} for the definition).

\begin{thmA} \label{thm:criterionuc}
Let $X$ be a uniformly curved Banach space. Then there exists an $\varepsilon(X) > 0$ such that the following holds: If $\Gamma$ is a group that admits a properly discontinuous cocompact action by simplicial automorphisms on a locally finite simplicial $2$-complex $M$ such that for all its links $L$, we have $\|A_L\|_{B(L^2_0(L,\nu))}<\varepsilon(X)$, then $\Gamma$ has property (F$_X$).
\end{thmA}
Theorem \ref{thm:criterionuc} provides a widely applicable criterion for fixed point properties for finitely presented groups, since such groups naturally act on the Cayley complex associated with the presentation. For general uniformly curved spaces, Theorem \ref{thm:criterionuc} relies on Pisier's renorming theorem \cite{MR0394135}, which states that superreflexive spaces can be renormed in such a way that a generalization (see \eqref{eq:p-uniform_convexity}) of the classical variance formula $\mathrm{Var}(U) = \mathbb{E}[U^2] - \mathbb{E}[U]^2$ holds for $X$-valued random variables. However, in many concrete examples of spaces $X$, our proof is entirely self-contained and we have explicit estimates for $\varepsilon(X)$ (see Remark \ref{rem:explicit_computation_of_epsilonX}).

Theorem \ref{thm:criterionuc} is a direct analogue of \.{Z}uk's spectral criterion mentioned above, since the condition $\|A_L\|_{B(L^2_0(L,\nu))} < \varepsilon$ means that the spectrum of $A_L$, apart from a simple eigenvalue $1$, is contained in $(-\varepsilon,\varepsilon)$, or equivalently, that the spectrum of the Laplacian on $L$, apart from a simple eigenvalue $0$, is contained in $(1-\varepsilon,1+\varepsilon)$. This condition can be viewed as a two-sided spectral gap.

One significant advantage of our spectral criterion is that it is entirely Hilbertian: It only relies on the eigenvalues of the \mbox{($2$-)Laplacian}, although the conclusion is on fixed point properties on uniformly curved spaces. In fact, as a corollary, we obtain bounds on the first positive eigenvalue of the $p$-Laplacian for other values of $p$, by means of interpolation (see Theorem \ref{thm:p-laplacian}).

Theorem \ref{thm:criterionuc} follows from the following more general criterion for fixed point properties that we prove, which is formulated in terms of the norm of the Markov operator acting on vector-valued $L^p$-spaces.
\begin{thmA} \label{thm:maincriterion}
  Let $1<p<\infty$, and let $X$ be a superreflexive Banach space. Then there exists an $\varepsilon'=\varepsilon'(p,X)>0$ such that the following holds: If $\Gamma$ is a group that admits a properly discontinuous cocompact action by simplicial automorphisms on a locally finite simplicial $2$-complex $M$ such that for all its links $L$, we have $\|A_L\|_{B(L^p_0(L,\nu;X))}<\varepsilon'$, then $\Gamma$ has property (F$_X$).
\end{thmA}
In the proof of Theorem \ref{thm:maincriterion}, working with the Markov operator rather than the Laplacian makes a real difference, since one can use interpolation techniques.

The essential part of the proof of Theorem \ref{thm:maincriterion} is to derive a $p$-Poincar\'e inequality with small constant from the fact that the Markov operator has small norm. From that point, the result follows from the proof of the aforementioned result of Bourdon or from the result of Oppenheim. For completeness, we also present an elementary proof of the fact that Poincar\'e inequalities give rise to fixed points (see Theorem \ref{thm:ppoincaretofixedpoint}). The line of proof is similar to Oppenheim's proof, and we claim no originality at this point.

As mentioned above, the second aim of this article is to apply our spectral criterion to random groups in the triangular density model. In the general setting of uniformly curved spaces, we obtain the following result.
\begin{thmA} \label{thm:maintheoremrandomgroups}
Let $\eta \in (0,2)$. There is a constant $C > 0$ and a sequence $(u_m)$ of positive real numbers tending to $0$ such that the following holds: For every $m \in \N$ and $d \in (0,1)$ satisfying
\begin{equation} \label{eq:condition_on_d}
  d \geq \frac 1 3 + \frac{\log \log m - \log(2-\eta)}{3 \log m},
\end{equation}
with probability $\geq 1-u_m$, a group in $\cM(m,d)$ has property (F$_X$) for every uniformly curved space satisfying $\varepsilon(X) \geq \sqrt{\frac{Cm}{(2m-1)^{3d}}}$. In particular, the latter is the case when $\varepsilon(X) \geq \sqrt{\frac{C}{\log m}}$.
\end{thmA}
In the setting of $L^p$-spaces, $\theta$-Hilbertian spaces with $\theta=\frac{2}{p}$ or Banach spaces $\alpha$-isomorphic to a subquotient of a $\theta$-Hilbertian space with $\theta=\frac{2}{p}$, we obtain the following result.
\begin{corA} \label{cor:randomgroupslp}
Let $\eta \in (0,2)$. There is a sequence $(u_m)$ of positive real numbers tending to $0$ such that the following holds: For every integer $m$ and $d \in (0,1)$ satisfying
\[
  d \geq \frac 1 3 + \frac{\log \log m - \log(2-\eta)}{3 \log m},
\]
with probability $\geq 1-u_m$, a group in $\mathcal{M}(m,d)$ has
  \begin{enumerate}
  \item\label{item:FLp_quantitative} property (F$_{L^p}$) for every $p \in \left[2,\sqrt{\frac{(3d-1) \log m}{\eta + \log 2}}\right]$;
  \item\label{item:FX_quantitative} property (F$_{X}$) for every $X$ that is $\alpha$-isomorphic to a subquotient of a $\frac 2 p$-Hilbertian space with $p \in \left[2, \sqrt{\frac{(3d-1) \log m}{\eta+\log (2\alpha)}} - \frac{1}{2}\right]$.
  \end{enumerate}
\end{corA}
Equation \eqref{item:FLp_quantitative} is an improvement of the results by Dru\c{t}u and Mackay. Indeed, in \cite[Corollary 1.7]{drutumackay}, an analogous statement was proved for $p$ in the range $\left[2,C\left( \frac{\log m}{\log \log m}\right)^{\frac{1}{2}} \right]$ and for $d>\frac{1}{3}$ independent from $m$, with weaker results for a slightly larger range of $d \geq \frac{1}{3}$ (see also \cite[Remark 9.5]{drutumackay}). Corollary \ref{cor:randomgroupslp} is even new for $p=2$ (corresponding to property (T)). Indeed, before it was only known that random groups in $\mathcal{M}(m,d)$ have property (T) with probability $1-o(1)$ in the regime
\[
  d \geq \frac 1 3 + \frac{\log \log m +\zeta}{3 \log m}
\]
for some constant $\zeta > 0$ (see \cite{MR3305311}).

Our results on fixed point properties of actions on $L^p$-spaces naturally lead to new estimates on the conformal dimension. The conformal dimensional $\mathrm{Confdim}(\partial_{\infty}\Gamma)$ of the boundary $\partial_{\infty}\Gamma$ of a hyperbolic group $\Gamma$ is a canonically defined quasi-isometry invariant, which was introduced by Pansu \cite{MR1024425} (see also \cite{mackaytyson}). Relying on a result of Bourdon (see \cite{MR3497258}), we obtain, along the same lines as Dru\c{t}u and Mackay, the following strengthening of the lower bound in \cite[Theorem 1.11]{drutumackay} as an immediate consequence of Corollary \ref{cor:randomgroupslp}.
\begin{corA}
Let $\eta \in (0,2)$. There is a sequence $(u_m)$ of positive real numbers tending to $0$ such that the following holds: For every integer $m$ and $d \in (0,1)$ satisfying
\[
  d \geq \frac 1 3 + \frac{\log \log m - \log(2-\eta)}{3 \log m},
\]
with probability $\geq 1-u_m$, a group in $\mathcal{M}(m,d)$ is hyperbolic and satisfies
\[
  \left(\frac{(3d-1)\log{m}}{\eta+\log 2}\right)^{\frac{1}{2}} \leq \mathrm{Confdim}(\partial_{\infty}(\Gamma)).
\]
\end{corA}

\subsection{Relation to other work}
In an earlier version of this article, we applied our criterion, i.e.~Theorem \ref{thm:criterionuc}, to another model of random groups. Indeed, we followed the approach of \cite{MR3106728} and obtained the following result, which is much weaker than Theorem \ref{thm:maintheoremrandomgroups}.
\begin{prp} \label{prp:kk}
Let $X$ be a uniformly curved Banach space. For every density $d > \frac{1}{3}$, a random group in the triangular model $\mathcal{M}(m,d)$ has property (F$_X$) w.o.p., that is
\[
  \lim_{m \to \infty} \mathbb{P}(\Gamma \textrm{ in } \mathcal{M}(m,d) \textrm{ has (F$_X$)})=1.
\]
\end{prp}
The novelty of this result was mainly the fact that we obtained fixed point properties for actions of random groups on Banach spaces very different from $L^p$-spaces, but in the setting of $L^p$-spaces, Proposition \ref{prp:kk} is weaker than the results of \cite{drutumackay}.

After this earlier version, it was suggested to us by Dru\c{t}u and Mackay to work with the binomial triangular model for random groups and Erd\H{o}s-R\'enyi graphs rather than following the approach of \cite{MR3106728}. They expected (see \cite[Section 1.4]{drutumackay}) that this would give results that are quantitative improvements of their results, which, as explained in Section \ref{subsec:mainresults}, indeed turns out to be the case.

In \cite{drutumackay}, Dru\c{t}u and Mackay also obtained results for random groups in the Gromov density model and for the critical density $d=\frac{1}{3}$. We expect that our results can be transferred to these settings as well.

\subsection{Organization of the article}
The article is organized as follows. Section \ref{sec:banachspaces} covers some preliminaries on the geometry of Banach spaces. In Section \ref{sec:graphs}, we explain how small Markov operators give rise to Poincar\'e inequalities. This section also includes some new results that may be of independent interest. In particular, we prove some bounds on the eigenvalues of $p$-Laplacians. In Section \ref{sec:poincarefpp}, we explain how Poincar\'e inequalities give rise to fixed points. Theorem \ref{thm:maincriterion} and Theorem \ref{thm:criterionuc} are proved in Section \ref{sec:proofmainresult}. In order to prove Theorem \ref{thm:maintheoremrandomgroups}, we first establish some new results on the spectrum and degree distribution of Erd\H{o}s-R\'enyi graphs in Section \ref{sec:ergraphs}. These may be of independent interest. Fixed point properties for random groups are investigated in Section \ref{sec:randomgroups}. In particular, Theorem \ref{thm:maintheoremrandomgroups} and Corollary \ref{cor:randomgroupslp} are proved in that section.

\section*{Acknowledgements}
We are indebted to Cornelia Dru\c{t}u and John Mackay for their suggestion to work in the binomial triangular model and with Erd\H{o}s-R\'enyi graphs (as described above). We also thank them for interesting discussions.

We thank Marc Bourdon and Gilles Pisier for useful comments on an earlier version of this article.

\section{Preliminaries on Banach spaces} \label{sec:banachspaces}

\subsection{Superreflexivity and uniform convexity} \label{subsec:sruc}
Two Banach spaces $X$ and $Y$ are called $C$-isomorphic if there exists an isomorphism $T\colon X \rightarrow Y$ such that $\|T\| \|T^{-1}\|\leq C$. The Banach-Mazur distance $d(X,Y)$ between $X$ and $Y$ is defined as the infimum of such $C$, where the infimum is taken over all linear isomorphisms between $X$ and $Y$.

A Banach space $Y$ is said to be finitely representable in a Banach space $X$ if for every finite-dimensional subspace $U$ of $Y$ and every $\varepsilon > 0$, there exists a subspace $V$ of $X$ such that $d(U,V) < 1 + \varepsilon$, where $d$ is the Banach--Mazur distance. A Banach space $X$ is called superreflexive if every Banach space that is finitely representable in $X$ is reflexive. Equivalently, a Banach space $X$ is superreflexive if and only if all its ultrapowers are reflexive.

A Banach space $X$ is uniformly convex if
\[
  \delta_X(\varepsilon) := \inf\left\{ 1 - \frac{\|x+y\|}{2} \;\Bigg\vert\; \|x\| \leq 1,\;\|y\| \leq 1,\;\|x-y\| \geq \varepsilon \right\} > 0
\]
for all $\varepsilon \in (0,2]$. The function $\delta_X$ is called the modulus of convexity of $X$.

Every uniformly convex Banach space is superreflexive, and every superreflexive Banach space admits an equivalent uniformly convex norm \cite{MR0336297}.

Let $p \in [2,\infty)$. A Banach space $X$ is called $p$-uniformly convex if there exists a $C>0$ such that $\delta_X(\varepsilon) \geq C\varepsilon^p$ for all $\varepsilon \in (0,2]$. Equivalently \cite[Proposition 2.4]{MR0394135} (see also \cite[Lemma 6.5]{MR3210176}), $X$ is $p$-uniformly convex if there exists a constant $C>0$ such that for every $X$-valued random variable $U$,
\begin{equation} \label{eq:p-uniform_convexity}
  \|\mathbb{E}[U]\|^p + C\,\mathbb{E}[\|U-\mathbb{E}[U]\|^p] \leq \mathbb{E}[\|U\|^p].
\end{equation}
It is in the form of \eqref{eq:p-uniform_convexity} that we will use $p$-uniform convexity. This inequality provides quantitative estimates on the error term in Jensen's inequality $\|\mathbb{E}[U]\|^p \leq \mathbb{E}[ \|U\|^p]$. It can be seen as a generalization of the classical variance equality for Hilbert space-valued random variables:
\[ \mathbb{E}[\|U-\mathbb{E}[U]\|^2] = \mathbb{E}[\|U\|^2] - \|\mathbb{E}[U]\|^2.\]

By a famous theorem of Pisier \cite{MR0394135}, every uniformly convex Banach space has an equivalent norm with respect to which it is $p$-uniformly convex for some $p \in [2,\infty)$. 

It is well known that if $X$ is an $L^p$-space with $p \geq 2$ or, more generally, a strictly $\theta$-Hilbertian space with $\theta=\frac{2}{p}$ (see Section \ref{subsec:thetahilbertian}), then $X$ is $p$-uniformly convex. In the case of strictly $\theta$-Hilbertian spaces, \eqref{eq:p-uniform_convexity} holds with $C=4^{1-\frac{1}{\theta}}$ (Proposition \ref{prp:p-uniform_constant_for_theta_hilbertian}).

\subsection{Complex interpolation}\label{subsection:complex_interpolation}
We refer to \cite{MR0482275} and \cite{MR2732331} for details on complex interpolation for compatible couples of (complex) Banach spaces. We recall that a compatible couple $(X_0,X_1)$ of Banach spaces is a pair of Banach spaces together with continuous linear embeddings from $X_0$ and $X_1$ into the same topological vector space $\mathcal{X}$, which can always be assumed to be a Banach space. Complex interpolation is a way to assign to such a couple $(X_0,X_1)$ a family $(X_\theta)_{\theta \in [0,1]}$ of Banach spaces (subspaces of $\mathcal X$) that interpolate between $X_0$ and $X_1$. For example, if $(\Omega,\mu)$ is a measure space and $(X_0,X_1) = (L^{p_0}(\Omega,\mu),L^{p_1}(\Omega,\mu))$ (seen as subspaces of the topological vector space of all measurable maps from $\Omega$ to $\C$), then $X_\theta$ is the space $L^{p_\theta}(\Omega,\mu)$, where $\frac{1}{p_\theta} = \frac{1-\theta}{p_0}+\frac{\theta}{p_1}$. More generally, if $(X_0,X_1)$ is a compatible couple, then the complex interpolation space of parameter $\theta$ for the couple $(L^{p_0}(\Omega,\mu;X_0),L^{p_1}(\Omega,\mu;X_1))$ is $L^{p_\theta}(\Omega,\mu;X_\theta)$. A fundamental property of complex interpolation is Stein's interpolation theorem, which roughly says the following: If $(X_0,X_1)$ and $(Y_0,Y_1)$ are compatible couples and if there is a holomorphic family of linear operators $T_z:X_0 + X_1 \to Y_0 + Y_1$, with $\mathrm{Re}(z) \in [0,1]$, such that for all $z$ with $\mathrm{Re}(z) \in \{0,1\}$, we have $\|T_z \colon X_{\mathrm{Re}(z)} \to Y_{\mathrm{Re}(z)} \| \leq M_{\mathrm{Re}(z)}$, then $\|T_z \colon X_{\mathrm{Re}(z)} \to Y_{\mathrm{Re}(z)} \| \leq M_{\mathrm{Re}(z)}$ for all $z$ with $\mathrm{Re}(z) \in [0,1]$, where $M_\theta = M_0^{1-\theta} M_1^\theta$.

\subsection{$\theta$-Hilbertian spaces} \label{subsec:thetahilbertian}
A strictly $\theta$-Hilbertian space is a Banach space that can be written as an interpolation space $(X_0,X_1)_{\theta}$, where $X_1$ is a Hilbert space and $\theta \in (0,1]$ (see \cite{MR555306}). $L^p$-spaces are clearly strictly $\theta$-Hilbertian: If $X$ is an $L^p$-space with $p \geq 2$, then $X=(X_0,X_1)_{\theta}$, where $X_0=L^\infty$, $X_1 = L^2$ and $\theta = \frac{2}{p}$.

In the setting of strictly $\theta$-Hilbertian spaces, we can derive \eqref{eq:p-uniform_convexity} with an explicit constant $C$.
\begin{prp} \label{prp:p-uniform_constant_for_theta_hilbertian}
  If $X$ is isometric to a subquotient of a strictly $\theta$-Hilbertian space, then \eqref{eq:p-uniform_convexity} holds with $C=4^{1-\frac{1}{\theta}}$. 
\end{prp}
\begin{proof}
  If $Y$ is a subquotient of $X$, then the best constant in \eqref{eq:p-uniform_convexity} is smaller for $Y$ than for $X$. Therefore, it is sufficient to consider the case when $X$ is strictly $\theta$-Hilbertian. Consider a complex interpolation space $X = (X_0,X_1)_\theta$ between a Hilbert space $X_1$ and an arbitrary Banach space $X_0$, continuously embedded into the same Banach space $\mathcal{X}$. Fix a probability space $(\Omega,\mu)$, and consider the holomorphic family $T_z \colon U\in L^1(\Omega;\mathcal{X}) \mapsto (\mathbb E[U], 2^{z-1} (U-\mathbb E[U])) \in \mathcal{X} \oplus L^1(\Omega;\mathcal{X})$. If $\mathrm{Re}(z)=0$, then we have $\| T_z \colon L^\infty(\Omega;X_0) \to X_0 \oplus_\infty L^\infty(\Omega;X_0)\| \leq 1 $, and for $\mathrm{Re}(z)=1$, we have $\| T_z \colon L^2(\Omega;X_1) \to X_1 \oplus_2 L^2(\Omega;X_1)\| \leq 1$. By the results recalled in Section \ref{subsection:complex_interpolation}, the complex interpolation space of parameter $\theta$ between $L^\infty(\Omega;X_0)$ and $L^2(\Omega;X_1)$ (respectively $X_0 \oplus_{\infty} L^{\infty}(\Omega;X_0)$ and $ X_1 \oplus_2 L^2(\Omega;X_1)$) is $L^{p_\theta}(\Omega;X_\theta)$ (respectively $ X_\theta \oplus_{p_\theta} L^{p_\theta}(\Omega;X_\theta)$), where $\frac{1}{p_\theta} = \frac{\theta}{2}$. By Stein's interpolation theorem, we have $\|T_\theta\colon L^{p_\theta}(\Omega;X_\theta) \to X_\theta \oplus_{p_\theta} L^{p_\theta}(\Omega;X_\theta)\| \leq 1$. This is exactly \eqref{eq:p-uniform_convexity} with constant $C=\frac{1}{2^{p_\theta-2}} = 4^{1-\frac 1 \theta}$.
\end{proof}
More generally, one can consider the class of $\theta$-Hilbertian spaces, as introduced by Pisier in \cite{MR2732331}, which is a natural class of Banach spaces that includes the strictly $\theta$-Hilbertian spaces, but also certain interpolation spaces between compatible families (rather than couples) of Banach spaces as well as there ultraproducts. Every result that we mention for strictly $\theta$-Hilbertian spaces can be extended to the class of $\theta$-Hilbertian spaces by considering complex interpolation for families of Banach spaces.

\subsection{Uniform curvedness}\label{subsection:unif_curved}
The notion of uniformly curved Banach space was introduced by Pisier in \cite{MR2732331}. Let $X$ be a Banach space, and let $T\colon L^2(\Omega_1,\mu_1) \to L^2(\Omega_2,\mu_2)$ an operator. If $T \otimes \mathrm{id}_X$ extends to a bounded operator $T_X$ from $L^2(\Omega_1,\mu_1;X)$ to $L^2(\Omega_2,\mu_2;X)$, then we denote by $\|T_X\|$ its norm. Otherwise, we set $\|T_X\|=\infty$. For a Banach space $X$, we set $\Delta_X(\varepsilon)=\sup \|T_X\|$, where the supremum is taken over all measure spaces $(\Omega_1,\mu_1)$ and $(\Omega_2,\mu_2)$ and operators ${T:L^2(\Omega_1,\mu_1) \to L^2(\Omega_2,\mu_2)}$ satisfying $\|T:L^1(\Omega_1,\mu_1) \to L^1(\Omega_2,\mu_2)\| \leq 1$, $\|T:L^{\infty}(\Omega_1,\mu_1) \to L^{\infty}(\Omega_2,\mu_2)\| \leq 1$ and $\|T:L^2(\Omega_1,\mu_1) \to L^2(\Omega_2,\mu_2)\| \leq \varepsilon$.

\begin{dfn}
  A Banach space $X$ is uniformly curved if $\Delta_X(\varepsilon) \to 0$ when $\varepsilon \to 0$.
\end{dfn}
Pisier proved that uniformly curved spaces are superreflexive \cite{MR2732331}, and hence, by the results recalled in Section \ref{subsec:sruc}, every uniformly curved space has an equivalent $p$-uniformly convex norm for some $p \in [2,\infty)$. Pisier also showed that the Banach spaces $X$ for which $\Delta_X(\varepsilon) = \cO(\varepsilon^\alpha)$ for some $\alpha>0$ are exactly the spaces that are isomorphic to a subquotient of a $\theta$-Hilbertian space for some $\theta>0$.

\section{Graphs, eigenvalues and Poincar\'e inequalities} \label{sec:graphs}
\subsection{$p$-Poincar\'e inequalities}
In this article, we consider unoriented connected graphs, potentially weighted. Thus, a graph is a pair $\cG=(V,\omega)$, where $V$ is a set (the vertices) and $\omega \colon V \times V \to \mathbb{R}_{+}$ is a function (the weight function) such that $\omega(s,t)=\omega(t,s)$ for every $s,t \in V$. A finite graph is a graph for which $V$ is finite. Unweighted graphs correspond to the case when $\omega$ takes values in $\{0,1\}$, in which case $\omega$ is the indicator function of the edge set. The degree of a vertex $s$ in a graph $\mathcal{G}=(V,\omega)$ is defined as the number $d_{\omega}(s)=\sum_{t \in V} \omega(s,t)$.

Let $\mathcal{G}=(V,\omega)$ be a finite graph. Equip $V \times V$ with the probability measure $\mathbb{P}(s,t) = \frac{\omega(s,t)}{\sum_{s',t' \in V} \omega(s',t')}$ and $V$ with the probability measure $\nu(s) = \frac{d_{\omega}(s)}{\sum_t d_{\omega}(t)}$. Note that $\nu$ is the stationary probability measure for the random walk on $\mathcal{G}$ with transition probability $p(s\to t) = \frac{\omega(s,t)}{d_\omega(s)}$. It is also the pushforward measure of $\mathbb{P}$ under both maps $(s,t) \mapsto s$ and $(s,t)\mapsto t$.

The gradient $\nabla f:V \times V \to X$ of a function $f:V \to X$ is defined by $(\nabla f)(e)=f(t)-f(s)$ if $e=(s,t)$. 
\begin{dfn}
  Let $\cG=(V,\omega)$ be a connected finite graph, and let $1 < p < \infty$. For a Banach space $X$, we denote by $\pi_{p,\mathcal{G}}(X)$ the smallest real number $\pi$ such that for all $f \colon V\to X$, the inequality
\[
  \inf_{x \in X} \|f-x\|_{L^p(V,\nu;X)} \leq \pi \|\nabla f\|_{L^p(V \times V,\mathbb{P};X)}
\]
holds. We call $\pi_{p,\mathcal{G}}(X)$ the $X$-valued $p$-Poincar\'e constant of $\mathcal{G}$.
\end{dfn}
Let $(Z_0,Z_1,\dots)$ be the random walk on $\mathcal{G}$ with $Z_0$ (and hence $Z_n$ for all $n \geq 0$) distributed as $\nu$. In this setting, the $X$-valued $p$-Poincar\'e constant of $\mathcal{G}$ is the smallest real number $\pi$ such that for all $f \colon V\to X$, the following inequality holds:
\[
  \inf_{x \in X} \mathbb{E}[\|f(Z_0)-x\|^p] \leq \pi^p \mathbb{E}[\|f(Z_0) - f(Z_1)\|^p].
\]
\begin{rem}
  The validity of the above inequality for every $f$ with $\mathbb E[\|f(Z_0)\|^p]<\infty$ can be used to define the $X$-valued $p$-Poincaré constant of an arbitrary irreducible Markov chain $(Z_0,Z_1,\dots)$ on a set $V$ with $\sigma$-finite stationary measure $\nu$. In the case when the measure $\nu$ is infinite (the terminology is that the Markov chain is not positively recurrent), the definition of $p$-Poincar\'e constant becomes simpler: It is the smallest $\pi$ such that for every $f \in L^p(V,\nu;X)$,
\[
  \mathbb{E}[\|f(Z_0)\|^p] = \int \|f\|_X^p d\nu \leq \pi^p \mathbb{E}[\|f(Z_0) - f(Z_1)\|^p].
\]
All results in this section hold in this generality, except for the interpretation in terms of the eigenvalues of ($p$-)Laplacians, where one needs reversibility of the Markov chain. The only adaptation in the proof of Theorem \ref{thm:improvement_triangle_inequality} when $\nu$ is infinite is that in that case, $L^p_0(V,\nu;X)$ is replaced by $L^p(V,\nu;X)$.
\end{rem}

Let us point out that our $p$-Poincar\'e constant differs (by a factor or power) from the $p$-Poincar\'e constants in \cite{MR2978328} and \cite{MR3323202}, neither does it exactly coincide with the conventions of \cite{MR3210176}.

Let $\mathcal{G}=(V,\omega)$ be a connected finite graph with $|V|=n$. We denote by $A_{\cG}$ (or simply $A$ if no confusion should arise) the Markov operator of the random walk on $\mathcal{G}$, which acts on functions on $V$ by the formula
\[
  A_\cG f(s) = \frac{1}{d_{\omega}(s)} \sum_{t \in V} \omega(s,t)f(t) = \mathbb E[f(Z_1) \mid Z_0=s].
\]
The operator $A_\cG$ is self-adjoint on $L^2(V,\nu)$. Indeed for $f,g \in L_2(V,\nu)$, we have
\begin{equation} \label{eq:bilinear_A_w}
  \langle A_\cG f,g \rangle = \frac{1}{\sum_{u \in V} d_{\omega}(u)}\sum_{s,t \in V} f(s) \omega(s,t) \overline{g(t)}.
\end{equation}
We denote by $\|A_\cG^0\|$ the norm of the restriction of $A_\cG$ to the orthogonal complement of the constant functions in $L^2(V,\nu)$. We denote the eigenvalues of $A_\cG$ by $\mu_1(A_\cG) \geq \ldots \geq \mu_n(A_\cG)$. The largest eigenvalue is $1$.

The (normalized) Laplacian on $\mathcal{G}$ is the operator $\Delta_{2,\omega} = \mathrm{Id}-A_\cG$, which maps a function $f$ to
\[
  \Delta_{2,\omega}f(s) = f(s) - \frac{1}{d_{\omega}(s)} \sum_{t \in V} \omega(s,t)f(t) = \mathbb E[f(Z_0)-f(Z_1) \mid Z_0=s].
\]
The following result summarizes some elementary properties of the $p$-Poincar\'e constant.
\begin{prp}\label{prp:basic_properties_of_Poincare} For every connected finite graph $\mathcal{G}=(V,\omega)$ and $1 < p < \infty$,
  \begin{enumerate}[(i)]
  \item $\pi_{p,\mathcal{G}}(X)  \geq \frac 1 2$ for every $X$ if $\mathcal{G}$ has at least two vertices,
  \item $\pi_{p,\mathcal{G}}(L^p) = \pi_{p,\mathcal{G}}(\C)$,
  \item $\pi_{2,\mathcal{G}}(\C)$ is equal to $\frac{1}{\sqrt{2-2\mu_2(A)}}$.
    \end{enumerate}
  \end{prp}
  \begin{proof}
We always have $\|\nabla f\|_{L^p(V \times V,\mathbb P;X)} \leq 2 \inf_{x \in X} \|f-x\|_{L^p(V,\nu;X)}$ by the triangle inequality, and therefore $\pi_{p,\mathcal{G}}(X) \geq \frac{1}{2}$ if $\mathcal{G}$ has at least two vertices. Assertion (ii) follows from Fubini's theorem. Assertion (iii) is classical and follows from the equality $\mathbb E[\|f(Z_0) - f(Z_1)\|^2] = 2 \langle(1-A) f,f\rangle$, which holds for all $f \in L^2(V,\nu)$.
\end{proof}

For $1 < p < \infty$, the constant $\pi_{p,\mathcal{G}}(\C)$ is related to the eigenvalues of the $p$-Laplacian (see Section \ref{subsec:p_laplacian}).
  
\subsection{From small Markov operators to Poincar\'e inequalities} \label{subsec:markovpoincare}
The validity of a $p$-Poincar\'e inequality is a very robust property (see \cite{MR3356758,MR3208067}). Indeed, it is obvious that if a Banach space $X$ is at Banach-Mazur distance $C$ from another Banach space $Y$, then $\pi_{p,\mathcal{G}}(X)  \leq C \pi_{p,\mathcal{G}}(Y)$. Also, the $p$-Poincar\'e inequality implies the validity of a $q$-Poincar\'e inequality for all $q<\infty$ (see Section \ref{subsec:matousek}).

For applications to fixed point properties, the crucial point is to prove that $\pi_{p,\mathcal{G}}(X)<1$. The aforementioned fact is therefore not useful, because the property ``$\pi_{p,\mathcal{G}}(X)<1$'' is not robust. The next result, which is one of the main points in this article, expresses that the property ``$\pi_{p,\mathcal{G}}(X)<1$'' is a consequence of another property, which is actually very robust.

In what follows, $L^p_0(V,\nu;X) = \{f \in L^p(V,\nu;X) \mid \mathbb{E}[f(Z_0)]=0\}$.
\begin{thm} \label{thm:improvement_triangle_inequality}
Let $X$ be a $p$-uniformly convex Banach space. Then there exist $\varepsilon,\delta>0$ (depending on $X$) such that for every connected finite graph $\mathcal{G}=(V,\omega)$ the following holds: If $\|A_\cG\|_{B(L^p_0(V,\nu;X))} \leq \varepsilon$, then for every $f \in L^p_0(V,\nu;X)$, we have
\[
  \|f\|_{L^p_0(V,\nu;X)} \leq (1-\delta)\|\nabla f\|_{L^p(V \times V,\mathbb{P};X)}.
\]
In particular, $\|A_\cG\|_{B(L^p_0(V,\nu;X))} \leq \varepsilon \implies \pi_{p,\mathcal G}(X)\leq 1-\delta$.
\end{thm}
In the random walk notation, we have to prove that
\[
  \mathbb{E}[\|f(Z_0)\|^p] \leq (1-\delta)^p\;\mathbb{E}[\| f(Z_0)-f(Z_1)\|^p]
\]
for every $f \colon V \to X$ with $\mathbb{E}[f(Z_0)]=0$.\\

The proof is divided in several steps. The first one is standard.
\begin{lem}
  For every $f \in L^p_0(V,\nu;X)$, we have $\|f\|_p \leq (1-\varepsilon)^{-1} \|f-A_{\cG}f\|_p$.
\end{lem}
\begin{proof}
  If $A_\cG$ has norm $\leq \varepsilon$, then $(1-A_\cG)^{-1} = \sum_{n \geq 0} A_\cG^n$ has norm $\leq (1-\varepsilon)^{-1}$.
\end{proof}
The triangle inequality implies, without any condition on $X$, that $\|f - A_{\cG}f\|_p \leq \|\nabla f\|_p$, and hence $\|f\|_p \leq \frac{1}{1-\varepsilon} \|\nabla f\|_p$. This is, however, not strong enough. The next lemma improves this inequality.

Recall that since $X$ is $p$-uniformly convex, there exists a constant $C$ such that \eqref{eq:p-uniform_convexity} holds for every $X$-valued random variable $U$.
\begin{lem}
 For every $f \in L^p_0(V,\nu;X)$, we have
\[
  \mathbb E[\| (f-A_{\cG}f)(Z_0)\|^p] \leq \mathbb E [\|f(Z_0) - f(Z_1)\|^p] - C\;\mathbb E [\|f(Z_1) - A_{\cG}f(Z_0)\|^p].
\]
\end{lem}
\begin{proof}
Let $U$ be the $X$-valued random variable $f(Z_0) - f(Z_1)$. With this notation we have $\mathbb E[ U | Z_0 ] = f(Z_0) - A_{\cG}f(Z_0)$ and $U - \mathbb E[ U | Z_0] = A_{\cG}f(Z_0) - f(Z_1)$. So applying \eqref{eq:p-uniform_convexity} conditionally to $Z_0$ and then averaging with respect to $Z_0$ proves the lemma.
\end{proof}
\begin{proof}[Proof of Theorem \ref{thm:improvement_triangle_inequality}]
By the triangle inequality and the fact that $Z_1$ is distributed as $Z_0$, we have
\[
  (\mathbb E [\|f(Z_1) -A_{\cG}f(Z_0)\|^p])^{\frac 1 p} \geq \|f\|_p - \|A_{\cG}f\|_p \geq (1-\varepsilon) \|f\|_p.
\]
Taking into account the two lemmas above, we obtain
\[
  \|f\|_p^p \leq \frac{1}{(1-\varepsilon)^p}\left( \|\nabla f\|_p^p - C(1-\varepsilon)^p \|f\|_p^p\right),
\]
from which we deduce
\[
  \|f\|_p \leq \frac{1}{(1+C)^{\frac 1 p}(1-\varepsilon)} \|\nabla f\|_p.
\]
If $\varepsilon>0$ is small enough, so that $(1+C)^{\frac 1 p}(1-\varepsilon)>1$, then the theorem follows for $\delta = 1 - \frac{1}{(1+C)^{\frac 1 p}(1-\varepsilon)}$.
\end{proof}
 \begin{rem} \label{rem:constants}
  If $X$ is an $L^p$-space for some $p \geq 2$ (or more generally a subquotient of a $\theta$-Hilbertian space with $\theta = \frac 2 p$), then it follows from Proposition \ref{prp:p-uniform_constant_for_theta_hilbertian} and from the proof above that in that case, Theorem \ref{thm:improvement_triangle_inequality} holds when $(1+2^{2-p})^{\frac 1 p}(1-\varepsilon)>1$; for instance, as soon as $\varepsilon \leq \frac{2}{p 2^{p}}$.
\end{rem}
 \begin{rem}\label{rem:bipartite} Consider a bipartite connected finite graph $\mathcal{G} = (V,\omega)$. This means that $V$ can be partitioned as $V_1 \cup V_2$ in such a way that every edge connects $V_1$ to $V_2$. Equivalently, $-1$ is an eigenvalue of $A_{\cG}$. In particular, $\|A_{\cG}\|_{B(L^p_0(V,\nu;X))}=1$, so Theorem \ref{thm:improvement_triangle_inequality} does not say anything about $\mathcal{G}$. This is not an accident: If $\mathcal{G}$ is the complete bipartite graph with $2n$ vertices, in which both parts of the partition consist of $n$ vertices (so that $A_{\cG}$ has spectrum $\{-1,0,1\}$), it was shown in \cite[Proposition 11.6]{drutumackay} that $\pi_{p,\mathcal{G}}(\C)>1$ for all $n$ large enough if $p > \frac{\log 4}{\log 5 - \log 4}=6.21\ldots$. However, an easy adaptation of our proof yields (with the same $\varepsilon$ and $\delta$): If $\|A_{\cG}f\|_{L^p(V,\nu;X)} \leq \varepsilon \|f\|_{L^p(V,\nu;X)}$ for every $f \colon V \to X$ that has zero mean on $V_1$ and on $V_2$, then for every such $f$, we have
   \[
     \|f\|_{L^p(V,\nu;X)} \leq (1-\delta)\|\nabla f\|_{L^p(V \times V,\mathbb{P};X)}.
   \]
 In particular, if $\pi_{p,\mathcal{G}}^{\textrm{bipartite}}(X)$ is defined as the smallest constant $\pi$ such that the inequality
   \[
     \inf_{g \textrm{ is constant on each part}}\|f-g\|_{L^p(V,\nu;X)} \leq \pi \|\nabla f\|_{L^p(V \times V,\mathbb{P};X)}
   \]
   holds, then we have $\pi_{p,\mathcal{G}}^{\textrm{bipartite}}(X)\leq 1-\delta$.
   \end{rem}
   
\subsection{Matou{\v{s}}ek's extrapolation result} \label{subsec:matousek}
In this section, we elaborate on the robustness of the validity of Banach space valued $p$-Poincar\'e inequalities, as indicated at the beginning of Section \ref{subsec:markovpoincare}. By an argument of Matou{\v{s}}ek \cite{MR1489105} (see also \cite[Lemma 5.5]{MR2183280}), the classical $p$-Poincar\'e inequality implies the validity of a $q$-Poincar\'e inequality for all $q<\infty$, with a multiplicative loss ($\geq 4$) on the Poincar\'e constant. The following proposition provides a Banach space valued generalization of this result, which is essentially due to Cheng \cite{MR3503071} (see \cite{MR3208067,MR3356758} for related results). We state and prove a slightly different version of this result, which in particular does not depend on the maximal degree of the graph. This section is independent from and not needed in the rest of this article.
\begin{prp} \label{prp:Matousek}
For every $1\leq p,q<\infty$, there is a constant $C$ such that for every Banach space $X$ and every graph $\mathcal G$,
  \[
    \pi_{p,\mathcal{G}}(X) \leq C \pi_{q,\mathcal{G}}(X)^{\max(\frac q p,1)}.
  \] 
\end{prp}
\begin{proof}
The proof is an adaptation of the original argument by Matou{\v{s}}ek. Suppose that a graph $\cG$ has $X$-valued $q$-Poincar\'e constant equal to $\pi_q$. For $x \in X$ and $\alpha>0$, we set $\{x\}^\alpha = \|x\|^{\alpha- 1} x$ if $x \neq 0$ and $\{0\}^\alpha =0$. The map $M_{p,q}:L^p(\Omega,\mu;X) \to L^q(\Omega,\mu;X)$ defined by $M_{p,q}(f)(\omega) = \{f(\omega)\}^{\frac p q}$ is a version of the classical Mazur map for vector-valued $L^p$-spaces. The next lemma shows that it has the same regularity properties as in the classical case $X=\C$.
\begin{lem} For every $1 \leq p,q <\infty$, there exists a constant $C_{p,q} > 0$ such that
  \[
    \|M_{p,q}(f_1) - M_{p,q}(f_2)\|_{L^q(\Omega,\mu;X)} \leq C_{p,q}\|f_1 - f_2\|_{L^p(\Omega,\mu;X)}^{\min(\frac p q,1)}
  \]
  for every measure space $(\Omega,\mu)$, every Banach space $X$ and every two functions $f_1$ and $f_2$ in the unit ball of $L^p(\Omega,\mu;X)$.
  \end{lem}
\begin{proof}
For $\mathbb{R}$-valued functions the lemma is classical (see \cite{zbMATH02570125}). In particular, there exists a $C > 0$ (depending on $p$ and $q$) such that for all $S_1,S_2 \colon \Omega \to \R^+$ with $\|S_1\|_p \leq 1, \|S_2\|_p\leq 1$, we have
\begin{equation} \label{eq:mazur_real}
  \|S_1^{\frac p q} - S_2^{\frac p q}\|_q \leq C \|S_1 - S_2\|_p^{\min(\frac p q,1)}.
\end{equation}
Let $f_1,f_2$ be in the unit ball of $L^p(\Omega,\mu;X)$, and define $g_i$ in (the unit ball of) $L^q(\Omega,\mu;X)$ by $g_i = M_{p,q}(f_i)$. Let $\delta = \|f_1 - f_2\|_p$. We write $f_i = S_i U_i$ with $S_i(\omega) = \|f_i(\omega)\| $ and $U_i(\omega)$ in the unit sphere of $X$, so that $g_i = S_i^{\frac p q} U_i$. By the triangle inequality, we have $\|f_1 - f_2\|_p \geq \|S_1 - S_2\|_p$, and
\[
  \|S_2(U_1 - U_2)\|_p \leq \|S_2 U_1 - S_1 U_1 \|_p+\|S_1 U_1 - S_2 U_2\|_p \leq 2 \|f_1 - f_2\|_p.
\]
Writing $g_1-g_2 = (S_1^{\frac p q} - S_2^{\frac p q}) U_1 + S_2^{\frac p q}(U_1-U_2)$, we obtain
\[
  \|g_1-g_2\|_q \leq \|S_1^{\frac p q} - S_2^{\frac p q}\|_q + \| S_2^{\frac p q}(U_1 - U_2)\|_q.
\]
The first term is less than $C \delta^{\min(\frac p q,1)}$ by \eqref{eq:mazur_real}. We can view the second term as the norm of $U_1 - U_2$ in $L^q(\Omega,S_2^p\mu;X)$. If $q \leq p$, then this norm is less than the norm of $U_1-U_2$ in $L^p(\Omega,S_2^p\mu;X)$, i.e.~less than $2\delta$. If $q \geq p$, then by H\"older's inequality this norm is less than the geometric mean of its norm in $L^\infty$ and its norm is $L^p$, i.e.~less than $2 \delta^{\frac p q}$. The previous inequality therefore becomes
\[
  \|g_1 - g_2\|_q \leq (C+2)\delta^{\min(\frac p q,1)}.
\]
This proves the lemma, because $\delta$ was defined as $\|f_1 - f_2\|_p$.
\end{proof}
\noindent \emph{Proof of Proposition \ref{prp:Matousek} (continuation).}
Let $f \in L^p(V,\nu;X)$. We have to prove that
\[
  (\mathbb E\|f(Z_0)-f(Z_1)\|^p)^{\frac 1 p} \geq \frac{1}{C \pi_q^{\max(\frac{q}{p},1)}} \inf_{x \in X} \|f - x\|_p.
\]
By homogeneity, we may assume that $\inf_{x \in X} \|f - x\|_p = \frac 1 2$, and by replacing $f$ by $f-x$ for a suitable $x$, we may assume that $\|f\|_p\leq 1$.

Let $g= M_{p,q}(f)$. It has norm $\leq 1$ in $L^q$. By the previous lemma, we have
\[
  \frac{1}{2} = \inf_{x \in X, \|x\|\leq 1} \|f-x\|_p \leq C_{q,p} \inf_{x \in X, \|x\|\leq 1} \|g-x\|_q^{\min(\frac q p,1)}.
\]
In particular, there is a constant $c$ (depending on $p,q$) such that $\inf_{x \in X} \|g-x\|_q \geq c$. By definition of $\pi_q$, we have
\[
  \left(\mathbb E \|g(Z_0) - g(Z_1)\|^q \right)^{\frac 1 q} \geq \frac{c}{\pi_q}.
\]
By the previous lemma, we obtain
\[
  \frac{c}{\pi_q} \leq C_{p,q} \left(\mathbb E \|f(Z_0) - f(Z_1)\|^p \right)^{\frac 1 p \min(\frac p q, 1)},
\]
or equivalently,
\[
  (\mathbb E\|f(Z_0)-f(Z_1)\|^p)^{\frac 1 p} \geq \left(\frac{c}{C_{p,q}\pi_q}\right)^{\max(\frac q p,1)}.
\]
This concludes the proof of the result.
\end{proof}

\subsection{Eigenvalues of the $p$-Laplacian}\label{subsec:p_laplacian}
For $\alpha>0$ and $z \in \C$, we write $\{z\}^\alpha = |z|^{\alpha - 1} z$ if $z \neq 0$ and we extend the definition by continuity to $\{0\}^\alpha = 0$. 

Let $\mathcal{G}=(V,\omega)$ be a connected finite graph, and let $\nu$ be the measure on $V$ defined at the beginning of this section. If $1 < p < \infty$, then the $p$-Laplacian on $\mathcal{G}$ is the non-linear map $\Delta_p \colon \R^V\to \R^V$ defined by
\[
  \Delta_p f(s) = \frac{1}{d_\omega(s)} \sum_{t \sim s} \omega(s,t)\{f(s) - f(t)\}^{p-1},
\]
where $d_\omega(s)$ denotes the degree of a vertex $s$.

In our opinion, it would be more natural to define the $p$-Laplacian by $\Delta_p f(s) = \{ \frac{1}{d_\omega(s)} \sum_{t \sim s} \omega(s,t)\{f(s) - f(t)\}^{p-1}\}^{\frac{1}{p-1}}$, but we use the conventional definition here.

A scalar $\lambda \in \R$ is called an eigenvalue of $\Delta_p$ if there exists an $f \neq 0$ such that $\Delta_p f= \lambda \{f\}^{p-1}$. Bourdon proved \cite[Lemme 1.3]{MR2978328} that the eigenvalues of $\Delta_p$ coincide with the critical values of $f \mapsto \frac{\|\nabla f\|_p^p}{\|f\|^p_{L^p(V,\nu)}}$. In \cite[Proposition 1.2]{MR2978328}, he proved that the smallest nonzero eigenvalue $\lambda_{1,p}(\mathcal{G})$ of $\Delta_p$ is related to the smallest constant $\pi$ for which the inequality
\[
  \inf_{c \in \R} \|f-c\|_{L^p(V,\nu)} \leq \pi \|\nabla f\|_p
\]
holds, by the formula $\lambda_{1,p}(\mathcal{G}) = \frac{1}{2\pi^p}$. Hence, the crucial inequality $\pi<1$ corresponds to the inequality $\lambda_{1,p}(\mathcal{G})>\frac 1 2$.

In the case of $X=\C$, we can reformulate Theorem \ref{thm:improvement_triangle_inequality} in order to obtain spectral information on the $p$-Laplacian. 

\begin{thm} \label{thm:p-laplacian}
Let $\mathcal{G}=(V,\omega)$ be a connected finite graph, and let $p \geq 2$. If the spectrum of $\Delta_{2,\omega}$ is contained in $\{0\} \cup [1-\varepsilon,1+\varepsilon]$ for some $\varepsilon > 0$, then
\[
  \lambda_{1,p}(\mathcal{G}) \geq \left(1-2^{1 - \frac 2 p} \varepsilon^{\frac 2 p}\right)^p \left(\frac 1 2 + 2^{1-p}\right).
\]
In particular, $\lambda_{1,p}(\mathcal{G}) > \frac 1 2$ if $\varepsilon \leq 2 p^{-\frac{p}{2}} 2^{-\frac{p^2}{2}}$.
\end{thm}
\begin{proof} The assumption that the spectrum of $\Delta_{2,\omega}$ is contained in $\{0\} \cup [1-\varepsilon,1+\varepsilon]$ means that $A_\mathcal{G}$ has norm $\leq \varepsilon$ as an operator on $L^2_0(V,\nu)$, or equivalently, that the operator $A_\mathcal{G} - P$ has norm $\leq \varepsilon$ as an operator on $L^2(V,\nu)$, where $Pf = \int f d\nu$ is the projection onto the constant functions. Since $A_\mathcal{G} -P$ has norm $\leq 2$ as an operator on $L^\infty(V,\nu)$, by interpolation, this implies that $A_\mathcal{G}-P$ has norm $\leq 2^{1 - \frac 2 p} \varepsilon^{\frac 2 p}$ as an operator on $L^p(V,\nu)$. In particular, $A_\mathcal{G}$ has norm $\leq 2^{1 - \frac 2 p} \varepsilon^{\frac 2 p}$ on $L^p_0(V,\nu)$. By Theorem \ref{thm:improvement_triangle_inequality} and Remark \ref{rem:constants}, we obtain that for every $f \in L^p_0(V,\nu)$,
    \[ \|f\|_p \leq (1+2^{2-p})^{-\frac 1 p}(1-2^{1 - \frac 2 p} \varepsilon^{\frac 2 p})^{-1}  \|\nabla f\|_p.\]
    This implies that the Poincaré inequality holds with constant $(1+2^{2-p})^{-\frac 1 p}(1-2^{1 - \frac 2 p} \varepsilon^{\frac 2 p})^{-1}$. The proposition follows from the relationship between the Poincaré constant and $\lambda_{1,p}(\mathcal{G})$ alluded to above.
\end{proof}

\section{From Poincar\'e inequalities to fixed point properties} \label{sec:poincarefpp}
Recall that if $M=(M_0,M_1,M_2)$ is a simplicial $2$-complex and $m \in M_0$, then the link $L(m)$ is the graph $(V,\omega)$ with vertex set the set of $1$-simplices in $M_1$ containing $m$ and $\omega(s,t)$ is the number of $2$-simplices in $M_1$ having $s$ and $t$ as two of its faces. In the following, we give, as mentioned in the introduction, a direct proof of the fact that Poincar\'e inequalities give rise to fixed points. The approach is similar to the one of Oppenheim \cite{MR3278887} and we claim no originality. We have chosen to leave out some computations.
\begin{thm} \label{thm:ppoincaretofixedpoint}
  Let $1 < p < \infty$, let $X$ be a Banach space, and let $M$ be a connected and locally finite simplicial $2$-complex. Suppose that $\pi_{p,L(m)}(X) <1$ for every $m \in M_0$. If $\Gamma$ is a group that admits a properly discontinuous cocompact action by simplicial automorphisms on $M$, then $\Gamma$ has property (F$_X$).
\end{thm}
\begin{proof}
Suppose that $\Gamma \acts M$ is a group action by simplicial automorphisms that is properly discontinuous and cocompact. Let $\Xi_0$ denote a set of representatives of the $\Gamma$-orbits in $M_0$, and for a vertex $m \in M_0$, let $\Gamma_{m}$ denote the stabilizer of $m$. Then $\Xi_0$ is a finite set (resp.~$\Gamma_{m}$ is a finite group), because the action $\Gamma \acts M$ is cocompact (resp.~properly discontinuous).

For $m \in \Xi_0$, we set $a_m = \frac{\sum_{s,t \in V} \omega(s,t)}{|\Gamma_m|}$. Let $\mathcal{E}$ be the affine space of $\Gamma$-equivariant maps $\psi:M_0 \to X$, which is naturally identified with $\prod_{m \in \Xi_0} X^{\Gamma_m}$ and is, in particular, nonempty.
\begin{lem} \label{lem:eneq}
  For $\varphi,\psi \in \mathcal{E}$ and $p \in [1,\infty)$, we have
  \begin{multline} \label{eq:mag}
    \sum_{m \in \Xi_0} a_m \|(n_1,n_2) \mapsto \varphi(n_1) - \psi(n_2)\|_{L^p(V \times V,\mathbb P;X)}^p \\= \sum_{m \in \Xi_0} a_m \|n \mapsto \varphi(n)-\psi(m)\|_{L^p(V,\nu;X)}^p.
  \end{multline}
  We denote this quantity by $E(\varphi,\psi)^p$, or simply by $E(\varphi)^p$ when $\varphi=\psi$.
  Moreover, we have the inequality
  \begin{equation} \label{eq:triangle}
    E\left(\frac{\varphi+\psi}{2}\right) \leq E(\varphi,\psi).
  \end{equation}
\end{lem}
\begin{proof}[Proof of Lemma \ref{lem:eneq}]
If $\varphi = \psi$, then \eqref{eq:mag} is exactly \cite[Lemma 4.1]{MR2978328}. For $\varphi \neq \psi$, the same computation proves the equality.
 For \eqref{eq:triangle}, we decompose the function $(n_1,n_2) \mapsto \frac{\varphi(n_1)+\psi(n_1)}{2} - \frac{\varphi(n_2)+\psi(n_2)}{2}$ as $(n_1,n_2) \mapsto \frac{\varphi(n_1)-\psi(n_2)}{2} - \frac{\varphi(n_2)-\psi(n_1)}{2}$. By the triangle inequality, we obtain
 \[
   E\left(\frac{\varphi+\psi}{2}\right) \leq \frac{1}{2} \left(E(\varphi,\psi)+E(\psi,\varphi)\right).
 \]
 This is \eqref{eq:triangle}, because $E(\varphi,\psi) = E(\psi,\varphi)$.
\end{proof}
\noindent \emph{Proof of Theorem \ref{thm:ppoincaretofixedpoint} (continuation).} We now define a complete distance on $\mathcal{E}$ by
\[
  d(\varphi,\psi) = \left(\sum_{m \in \Xi_0} a_m  \|\varphi(m) - \psi(m)\|^p\right)^{\frac 1 p}.
\]
Take $c<1$ such that $\pi_{p,L(m)}(X)<c$ for every $m \in \Xi_0$, and let $\varphi \in \mathcal{E}$. By definition of $\pi_{p,L(m)}(X)$, for every $m \in \Xi_0$, there is a $\psi(m) \in X$ such that
\[
  \|n \mapsto \varphi(n)-\psi(m)\|_{L^p(V,\nu;X)} \leq c \|(n_1,n_2) \mapsto \varphi(n_1) - \varphi(n_2)\|_{L^p(V \times V,\mathbb{P};X)}.
\]
Moreover, by the $\Gamma$-equivariance of $\varphi$, the quantity $\|n \mapsto \varphi(n)-\psi(m)\|_{L^p(V,\nu;X)}$ is unchanged if $\psi(m)$ is replaced by an element of the $\Gamma_m$-orbit of $\psi(m)$. Therefore, by replacing $\psi(m)$ by the average on its $\Gamma_m$-orbit, we can assume that $\psi(m)$ is $\Gamma_m$-invariant. This means that $\psi$, defined so far only on $\Xi_0$, can be extended to an element of $\mathcal E$. Summing the $p$-th power of the previous expression over $m \in \Xi_0$, yields $E(\varphi,\psi) \leq c E(\varphi)$.
By \eqref{eq:triangle}, this implies $E\left(\frac{\varphi+\psi}{2}\right) \leq c E(\varphi)$.

On the other hand, using the triangle inequality, we obtain that for $m \in \Xi_0$,
\begin{multline*}
  \left\|\varphi(m) - \frac{\varphi(m)+\psi(m)}{2}\right\| = \frac 1 2 \|n \mapsto (\varphi(m) - \varphi(n)) + (\varphi(n) - \psi(m)) \|_p \\ \leq \frac 1 2 \|n \mapsto \varphi(m) - \varphi(n)\|_p + \frac 1 2\|n \mapsto \varphi(n) - \psi(m)\|_{p},
\end{multline*}
where $\|\cdot\|_p$ denotes the norm on $L^p(V,\nu;X)$. It follows that
\[
  d\left(\varphi, \frac{\varphi+\psi}{2}\right) \leq \frac{1}{2}(E(\varphi) + E(\varphi,\psi)) \leq E(\varphi).
\]

The conclusion of the preceding discussion is that for every $\varphi \in \mathcal{E}$, there is a $\varphi' \in \mathcal{E}$ (namely $\varphi' = \frac{\varphi+\psi}{2}$) such that $E(\varphi') \leq c E(\varphi)$ and $d(\varphi,\varphi') \leq E(\varphi)$. If we start from some $\varphi_0 \in \mathcal E$, by induction we obtain a sequence $(\varphi_n)$ in $\mathcal E$ with $E(\varphi_n) \leq c^n E(\varphi_0)$ and $d(\varphi_n,\varphi_{n+1}) \leq c^n E(\varphi_0)$. The sequence $(\varphi_n)$ is a Cauchy sequence and therefore converges to some $\varphi_\infty \in \mathcal E$ satisfying $E(\varphi_\infty)=0$. For a general complex $M$, the formula $E(\varphi_\infty)=0$ means that $\varphi_\infty$ is constant on the connected components of $L(m)$ for every $m \in M_0$. Here the assumption that $\pi_{p,L(m)}(X)<\infty$ implies that $L(m)$ is connected, and hence the assumption that $M$ is connected implies that $\varphi_\infty$ is constant and is necessarily equal to a fixed point. This proves the theorem.
\end{proof}
\begin{rem}
  As pointed out in \cite[Proposition 11.6]{drutumackay}, for bipartite graphs and large $p$, the condition $\pi_{p,\mathcal{G}}(\C) <1$ is in general not satisfied, even not for the complete bipartite graphs. As a consequence, Theorem \ref{thm:ppoincaretofixedpoint} is not applicable when $M$ has a link that is bipartite.

  \begin{qst} \label{question:bipartite}
    Does Theorem \ref{thm:ppoincaretofixedpoint} hold if the assumption $\pi_{p,L(m)}(X) <1$ is replaced by $\pi_{p,L(m)}^{\textrm{bipartite}}(X)<1$ when $L(m)$ is bipartite?
  \end{qst}

By Remark \ref{rem:bipartite} and the argument in the following section, a positive answer to this question would imply that Theorem \ref{thm:criterionuc} holds when, for bipartite links $L$, the assumption $\|A_L\|_{B(L^2_0(L,\nu))}<\varepsilon$ is replaced by ``$[\varepsilon,1[$ does not intersect the spectrum of $A_L$''. This would in particular improve the results of \cite{drutumackay} and imply that for every $p<\infty$, random groups in the Gromov density model with densities $>\frac{1}{3}$ have property (F$L^p$) with probability tending to $1$.
\end{rem}

\section{Proofs of Theorem \ref{thm:maincriterion} and Theorem \ref{thm:criterionuc}} \label{sec:proofmainresult}
If $X$ is $p$-uniformly convex, then Theorem \ref{thm:maincriterion} is a direct combination of Theorem \ref{thm:improvement_triangle_inequality} and Theorem \ref{thm:ppoincaretofixedpoint}. Moreover, if $X$ is an $L^p$-space with $p \geq 2$ (or more generally a subquotient of a $\theta$-Hilbertian space with $\theta = \frac 2 p$), then we see from Remark \ref{rem:constants} that Theorem \ref{thm:maincriterion} holds with $\varepsilon' = \frac{2}{p 2^{p}}$.

We now prove the general case of Theorem \ref{thm:maincriterion}. The idea of the proof is to reduce to the case of $p$-uniformly convex Banach spaces.
\begin{proof}[Proof of Theorem \ref{thm:maincriterion}]
Let $1 < p < \infty$, and let $X$ be a superreflexive Banach space. As was recalled in Section \ref{sec:banachspaces}, by a famous result of Pisier \cite{MR0394135}, there exists a $q \in [2,\infty)$ and an equivalent norm $N$ on $X$ that is $q$-uniformly convex. Pisier's proof has the feature that every isometry of $(X,\|\cdot\|)$ remains an isometry of $(X,N)$, but even if this were not the case, we could always assume this by replacing $N$ by the equivalent norm $N'(x) = \sup_{g \in O(X,\|\cdot\|)} N(g x)$ (see the proof of $(2)\implies(3)$ in \cite[Proposition 2.3]{MR2316269}). Denote the Banach space $(X,N)$ by $Y$. We now use the following interpolation result, which was already used in a similar context in \cite{MR2732331}. 
\begin{lem} \label{lem:interpolationA}
There exists a constant $C > 0$ and a $\theta \in (0,1]$ such that for every graph $\mathcal G=(V,\omega)$, we have
\[
  \| A_\mathcal G\|_{B(L^q_0(V,\nu;Y))} \leq C \| A_\mathcal G\|_{B(L^p_0(V,\nu;X))}^\theta.
\]
\end{lem}
\begin{proof} Let $r \in [1,\infty]$ and $\theta \in (0,1)$ such that $\frac 1 q = \frac \theta p + \frac{(1-\theta)}{r}$. The operator $f \mapsto A_{\mathcal{G}}(f - \int f d\nu)$ has norm less than $2\| A_\mathcal G\|_{B(L^s_0(V,\nu;X))}$ on $L^s(V,\nu;X)$, since this operator is the composition of the operator $f \mapsto f-\int fd\nu \in L^s_0$ with norm at most $2$ with the restriction of $A_\mathcal G$ to $L^s_0$. Therefore, by interpolation we have
\[
  \|A_\mathcal G\|_{B(L^q_0(V,\nu;X))} \leq 2 \| A_\mathcal G\|_{B(L^p_0(V,\nu;X))}^\theta \| A_\mathcal G\|_{B(L^r_0(V,\nu;X))}^{1-\theta},
\]
which we simply bound by $2 \| A_\mathcal G\|_{B(L^p_0(V,\nu;X))}^\theta$. The conclusion now follows, because $\| A_\mathcal G\|_{B(L^p_0(V,\nu;Y))}$ is less than the product of $\| A_\mathcal G\|_{B(L^p_0(V,\nu;X))}$ and the Banach-Mazur distance between $X$ and $Y$.
\end{proof}
\noindent \emph{Proof of Theorem \ref{thm:maincriterion} (continuation).}
Since $Y$ is $q$-uniformly convex, by the case already proved, there exists an $\varepsilon_1>0$ such that a group with a properly discontinuous cocompact action by simplicial automorphisms on a simplicial $2$-complex $M$ with all its links $L$ satisfying $\|A_L\|_{B(L^q_0(L,\nu;Y))}<\varepsilon_1$ has property (F$_Y$). Therefore, if $\varepsilon'>0$ satisfies $C \varepsilon'^\theta \leq \varepsilon_1$, then every such group has property (F$_Y$). In particular it has property (F$_X$) because by construction every action by affine isometries on $X$ is an action by affine isometries on $Y$.
\end{proof}

Finally, we explain how Theorem \ref{thm:criterionuc} follows from Theorem \ref{thm:maincriterion}.
\begin{proof}[Proof of Theorem \ref{thm:criterionuc}]
Let $X$ be a uniformly curved Banach space. As recalled in Section \ref{subsection:unif_curved}, the space $X$ is superreflexive. Let $\varepsilon'= \varepsilon'(2,X)$ be given by Theorem \ref{thm:maincriterion} for $p=2$. We claim that for a finite graph $\mathcal{G}$, we have the inequality
\begin{equation}\label{eq:consequence_unif_curvedness}
  \|A_{\mathcal G}\|_{B(L^2_0(V,\nu;X))} \leq 2 \Delta_X\left(\frac{1}{2}\|A_{\mathcal G}\|_{B(L^2_0(V,\nu))}\right).
\end{equation}
Indeed, let $T \colon L^2(V,\nu) \to L^2(V,\nu)$ be the operator $f \mapsto \frac 1 2(A_{\mathcal G} f - \int fd\nu)$. Then $T$ has norm $\leq 1$ on $L^1(V,\nu)$ and $L^\infty(V,\nu)$ and norm $\frac{1}{2}\|A_{\mathcal G}\|_{B(L^2_0(V,\nu))}$ on $L^2(V,\nu)$, so by definition of $\Delta_X$, we have $\|T_X\| \leq \Delta_X( \|T\|)$. We obtain \eqref{eq:consequence_unif_curvedness} by considering the restriction to $L^2_0(V,\nu;X)$. Since $X$ is uniformly curved, there is an $\varepsilon>0$ such that $2 \Delta_X(\frac{\varepsilon}{2})<\varepsilon'$. It follows from \eqref{eq:consequence_unif_curvedness} that Theorem \ref{thm:criterionuc} follows for this value of $\varepsilon$.
\end{proof}
\begin{rem}\label{rem:explicit_computation_of_epsilonX} In many cases, we can give a direct proof of Theorem \ref{thm:criterionuc} (not relying on \cite{MR0394135}) and compute the constants explicitly.
  \begin{enumerate}[(i)]
  \item If $X$ is an $L^p$-space with $p \geq 2$, then Theorem \ref{thm:criterionuc} holds with $\varepsilon = 2 p^{-\frac p 2} 2^{-\frac{p^2}{2}}$.
  \item If $X$ is stricly $\frac 2 p$-Hilbertian, or more generally a subquotient of a stricly $\frac{2}{p}$-Hilbertian space, then Theorem \ref{thm:criterionuc} also holds with $\varepsilon = 2 p^{-\frac p 2} 2^{-\frac{p^2}{2}}$.
  \item If $X$ is isomorphic to a space as in (i) or (ii), with Banach-Mazur distance $d$, then Theorem \ref{thm:criterionuc} holds with $\varepsilon = K p^{-\frac p 2}2^{-\frac{p^2}{2}} d^{-\frac{p(p+1)}{2}}$, where $K$ is a universal constant.
    \end{enumerate}
\end{rem}
\begin{proof}
  As above, we consider the operator $f \mapsto A_{\mathcal G}f - \int fd\nu$. If $X$ is a Hilbert space and $p=2$, then its norm on $L^2(V,\nu;X)$ is equal to $\|A_{\mathcal G}\|_{B(L^2_0(V,\nu))}$ (say by decomposing in an orthonormal basis), whereas for $p=\infty$ and $X$ arbitrary, we have the trivial bound $\|f \mapsto A_{\mathcal G}f - \int fd\nu\|_{B(L^\infty(V,\nu;X))} \leq 2$. So if $p \geq 2$ and $X$ is an $L^p$-space (or more generally a strictly $\frac 2 p$-Hilbertian space), then interpolation directly gives the inequality
\[
  \|f \mapsto A_{\mathcal G}f - \int fd\nu\|_{B(L^p(V,\nu;X))} \leq 2^{1 - \frac 2 p}\|A_{\mathcal G}\|_{B(L^2_0(V,\nu))}^{\frac 2 p}
\]
for every graph $\mathcal G$. So, by Remark \ref{rem:constants}, we obtain that Theorem \ref{thm:criterionuc} holds as soon as $2^{1- \frac 2 p} \varepsilon^{\frac 2 p} \leq \frac{2}{p2^p}$, i.e.~as soon as $\varepsilon \leq 2 p^{-\frac p 2} 2^{-\frac{p^2}{2}}$. The same argument works if $X$ is a subquotient of a stricly $\frac{2}{p}$-Hilbertian space. The additional argument is to observe that the quantity $\| f\in L^p(V,\nu;X) \mapsto A_{\mathcal G} f - \int f d\nu\|$ can only decrease when $X$ is replaced by a subquotient of $X$.

Finally, consider the case when $X$ is at Banach-Mazur distance $\leq d$ from a subquotient of a strictly $\frac 2 p$-Hilbertian space $Y$. First, suppose that \eqref{eq:p-uniform_convexity} holds for $X$ with $C = \frac{1}{2^{p-2}d^p}$. Second, we will reduce to this case. By the above proof for $Y$, for every finite graph $\mathcal G$, we have $\|A_{\mathcal G}\|_{B(L^p_0(V,\nu;Y))} \leq 2^{1- \frac 2 p}\|A_{\mathcal G}\|_{B(L^p_0(V,\nu))}^{\frac 2 p}$, and therefore,
\[
  \|A_{\mathcal G}\|_{B(L^p_0(V,\nu;X))} \leq 2^{1- \frac 2 p} d \|A_{\mathcal G}\|_{B(L^p_0(V,\nu))}^{\frac 2 p}.
\]
Hence, by the proof of Theorem  \ref{thm:improvement_triangle_inequality}, we have that $\pi_{p,\mathcal G}(X,N)<1$ as soon as $\|A_{\mathcal G}\|_{B(L^p_0(V,\nu))} \leq \varepsilon$ with $(1+\frac{1}{2^{p-2}d^p})^{\frac 1 p}(1-2^{1- \frac 2 p}d \varepsilon^{\frac 2 p})>1$. Elementary computations show that this holds as soon as $\varepsilon \leq K p^{-\frac p 2}2^{-\frac{p^2}{2}} d^{-\frac{p(p+1)}{2}}$, where $K$ is a universal constant.

To conclude the proof, we explain how we can reduce to the case in which \eqref{eq:p-uniform_convexity} holds for $X$ with $C = \frac{1}{2^{p-2}d^p}$. We identify $X$ and $Y$ in such a way that the norms $\|\cdot\|_X$ and $\|\cdot\|_Y$ satisfy $\|x\|_X \leq \|x\|_Y \leq d\|x\|_X$ for all $x \in X$. Denote by $O(X)$ the group of linear isometries of $(X,\|\cdot\|_X)$ and define the norm $N(x) = \sup_{g \in O(X)} \|gx\|_Y$, so that $\|x\|_Y \leq N(x) \leq d \|x\|_Y$ for all $x \in X$. By construction, every action by affine isometries on $(X,\|\cdot\|_X)$ is an action by isometries on $(X,N)$. We have to prove that for $(X,N)$ \eqref{eq:p-uniform_convexity}  holds with $C = \frac{1}{2^{p-2}d^p}$. This follows from Proposition \ref{prp:p-uniform_constant_for_theta_hilbertian}, which asserts that \eqref{eq:p-uniform_convexity} holds for $\|\cdot\|_Y$ with $C=\frac{1}{2^{p-2}}$. Indeed, if $U$ is an $X$-valued random variable, then for every $g \in O(X)$ we can apply \eqref{eq:p-uniform_convexity} to the random variable $gU$ and obtain
\[
  \|g \mathbb E U\|_Y^p + \frac{1}{2^{p-2}} \mathbb E \| g(U-\mathbb E U)\|_Y^p \leq \mathbb E\|g U\|_Y^p.
\]
In particular, by the inequality $N(\cdot)\leq d \|g \cdot\|_Y \leq N(\cdot)$, we obtain
\[
  \|g \mathbb E U\|_Y^p + \frac{1}{2^{p-2}d^p } \mathbb E N(g(U-\mathbb E U))^p \leq  \mathbb E N(U)^p.
\]
By taking the supremum over $g \in O(X)$ we obtain
\[
  N(\mathbb E U)^p + \frac{1}{2^{p-2}d^p } \mathbb E N(g(U-\mathbb E U))^p \leq  \mathbb E N(U)^p,
\]
which finishes the proof.
\end{proof}

\section{Erd\H{o}s-R\'enyi graphs: spectrum and degree distribution} \label{sec:ergraphs}
In this section, we collect and establish some results on the spectrum and degree distribution of Erd\H{o}s-R\'enyi graphs, which will be used in Section \ref{sec:randomgroups}.

For future reference, we first recall a form of Chernoff's inequality (see for example the first pages of \cite{MR1864966}), which provides standard concentration bounds for binomial random variables. If $S$ is a binomial random variable $B(N,\rho)$ (meaning that $S$ takes integer values $k \in \{0,1,\dots,N\}$ with probability $\binom{N}{k} \rho^k (1-\rho)^{N-k}$, then for every positive $c$, we have
\begin{align}\label{eq:chernoff_lower} \mathbb P[ S-\mathbb E S \leq -c \mathbb E S] &\leq 2e^{-\frac{c^2 \mathbb{E} S}{2}},\\
 \label{eq:chernoff_upper} \mathbb P[ S-\mathbb E S \geq c \mathbb E S] &\leq e^{-\frac{c^2 \mathbb{E} S}{2+c}}.\end{align}

\subsection{Erd\H{o}s-R\'enyi graphs}
If $m$ is a positive integer and $\rho \in [0,1]$, an Erd\H{o}s-R\'enyi graph $\mathbb{G}(m,\rho)$ is a random graph with $m$ vertices in which each unoriented edge $\{s,t\}$ with $s \neq t$ occurs independently with probability $\rho$.

It is well known (see e.g.~\cite{MR1864966}) that the connectivity threshold occurs at ${\rho \sim \frac{\log m}{m}}$: For every $\eta>0$, the probability that $\mathbb{G}(m,\rho)$ is connected is $1-o(1)$ if $\rho \geq  (1+\eta) \frac{\log m}{m}$ and $o(1)$ if $\rho \leq (1-\eta) \frac{\log m}{m}$.

We will need a lemma stating that, above the connectivity threshold, all vertices have degree of the same order.

\begin{lem} \label{lem:concentration_degree}
Let $\eta>0$. There are constants $c_1,c_2 > 0$ and a sequence $(u_m)$ of positive real numbers tending to $0$ such that the following holds: For every $\rho \geq (1+\eta) \frac{\log m}{m}$, with probability $\geq 1-u_m$, the degree of every vertex in an Erd\H{o}s-R\'enyi graph $\mathbb{G}(m,\rho)$ is in the interval $[c_1 m\rho,c_2m\rho]$.
\end{lem}
\begin{proof} The degree of every vertex in $\mathbb{G}(m,\rho)$ is a binomial $B(m-1,\rho)$ random variable. For $\eta \geq 100$, the lemma is a straightforward application of Chernoff's inequalities \eqref{eq:chernoff_lower} and \eqref{eq:chernoff_upper}, say with $c_1=\frac 1 2$ and $c_2= \frac 3 2$. For $\eta \in (0,100)$, the lemma follows for example from \cite[Exercise 3.4]{MR1864966}, which provides the optimal values for $c_1$ and $c_2$. We could not find a solution to this exercise in the literature, so for the reader's convenience, we provide a proof of the only nontrivial inequality, i.e. the lower bound (the upper bound for $c_2=3$ is also a direct application of Chernoff's inequality \eqref{eq:chernoff_upper}).

Suppose that $(1+\eta) \log m \leq m\rho \leq 101 \log m$. Let $c>0$ be a constant to be determined later. We have
\begin{align*}
  \mathbb P( B(m-1,\rho) \leq cm\rho) &\leq \sum_{0 \leq k \leq cm\rho} \rho^k (1-\rho)^{m-1-k} \binom{m}{k}\\& \leq \sum_{0 \leq k \leq cm\rho} \frac{(m\rho)^k}{k!} (1-\rho)^{m-1-k}
\end{align*}
by bounding $\binom m k \leq \frac{m^k}{k!}$. For $c$ small enough, we have that $\frac{(m\rho)^k}{k!} (1-\rho)^{m-k-1} \leq \frac 1 2\frac{(m \rho)^{k+1}}{(k+1)!} (1-\rho)^{m-k-2}$, so the whole sum is less than twice the last term. But if $k$ is the integer part of $cm\rho$, we have
  \[\log k! = k \log k -k + O(\log k) = cm\rho \log(cm\rho) -cm\rho + O(\log \log m),\]
  \[ \log (m\rho)^k = cm\rho \log(m\rho)+ O(\log \log m)\]
  and
  \[ \log (1-\rho)^{m-k} = -m\rho + O(\frac{1}{m}(\log m)^2).\]
  So
  \[ \log\left( \frac{1}{k!} (m\rho)^k (1-\rho)^{m-k} \right)  = m\rho(-1+c-c\log c) + O(\log \log m).\]
  In particular, if $c$ is small enough so that $(1+\eta) (1-c+c\log c) > 1+\frac \eta 2$, we obtain
  \[ \frac{1}{k!} (m\rho)^k (1-\rho)^{m-k} \leq \frac{1}{m^{1+\frac \eta 2}}\]
  for all $m$ large enough and
  \[ \mathbb P( B(m-1,\rho) \leq cm\rho) \leq \frac{2}{m^{1+\frac \eta 2}}.\]
  By a union bound we obtain
  \[ \mathbb P( \min_{v \in \mathbb{G}(m,\rho)} B(m-1,\rho) \leq cm\rho) \leq 2 m^{- \frac \eta 2}.\]
  This proves the lemma.
\end{proof}
By \cite[Exercise 3.4]{MR1864966}, it is easy to see that the previous lemma is essentially optimal, in the sense that for every $\eta>0$ and $\rho \sim (1+\eta) \frac{\log m}{m}$, the ratio maximal degree/minimal degree converges to a constant $r(\eta) > 1$ as $m \to \infty$, and that $\lim_{\eta \to 0} r(\eta) = \infty$. We shall need, however, that in $L^1$-average, the degree sequence is well concentrated. This is the content of the next lemma.
\begin{lem} \label{lem:average_concentration}
There is a constant $c_0 > 0$ such that the following holds. Let $d(1),\dots,d(m)$ be the degree sequence of an Erd\H{o}s-R\'enyi graph $\mathbb{G}(m,\rho)$, and let $\overline{d} = \frac{1}{m} \sum_i d(i)$. If $\frac{\log m}{m} \leq \rho \leq 1$, then 
\[
  \mathbb P\left[ \sum_i \frac{|d(i) - (m-1) \rho|}{m(m-1)\rho} \geq \frac{c_0}{\sqrt{(m-1)\rho}}\right] \leq e^{-\frac{m}{c_0}}
\]
and 
\[
  \mathbb P\left[ \sum_i \frac{|d(i) - \overline{d}|}{m\overline{d}} \geq \frac{c_0}{\sqrt{(m-1)\rho}}\right] \leq e^{-\frac{m}{c_0}}.
\]
\end{lem}
\begin{proof}
We will prove both equalities for $m$ sufficiently large, i.e.~we will prove that there exists a $c_0 > 0$ and an $m_0 \in \mathbb{N}$ such that for all $m \geq m_0$, the inequalities above hold. By the monotonicity properties of both inequalities, the inequalities then follow for all $m$, possibly after replacing $c_0$ by a larger constant.

Let us first prove the first inequality, the proof of which is similar to the proof of Chernoff's inequality: We obtain tight concentration bounds from exponential moments. First, note that there exists an $a > 0$ such that $e^x \leq 1 + x + ax^2$ for all $x \in [-2,2]$. It follows that for every $\rho \in (0,1)$ and $\lambda \in [-1,1]$, whenever $\varepsilon$ is a random variable equal to $1$ with probability $\rho$ and $0$ with probability $1-\rho$, then 
\[
  \mathbb E[e^{\lambda (\varepsilon-\mathbb E[\varepsilon])}] \leq 1 + a\rho(1-\rho)\lambda^2 \leq e^{a \rho \lambda^2}.
\]
As a consequence, if $Y$ is a binomial random variable $B(n,\rho)$, then
\begin{equation} \label{eq:exp_moment_Binomial}
  \mathbb E[e^{\lambda (Y-\mathbb E[Y])}] \leq e^{a \rho n \lambda^2}.
\end{equation}
Write $d(i) = d'(i)+d''(i)$, where $d'(i)$ is the number of edges between $i$ and a vertex $j<i$, and $d''(i)$ is the number of edges between $i$ and a vertex $j>i$. In this way, the random variables $d'(1),\ldots,d'(m)$ are independent, and so are the random variables $d''(1),\ldots,d''(m)$. Write $X'=\sum_i |d'(i) - \mathbb E[d'(i)]|$ and $X''=\sum_i |d''(i) - \mathbb E[d''(i)]|$. 

By \eqref{eq:exp_moment_Binomial}, it follows that for $\lambda \in [0,1]$, we have
\begin{eqnarray*}
  \mathbb E[e^{\lambda |d'(i) - \mathbb E[d'(i)]|}] &\leq& \mathbb E[e^{\lambda (d'(i) - \mathbb E[d'(i)])}]+\mathbb E[e^{-\lambda (d'(i) - \mathbb E[d'(i)])}] \\&\leq& 2 e^{a \rho \lambda^2(i-1)},
\end{eqnarray*}
and by independence, we have
\[
  \mathbb E[e^{\lambda X'}] \leq 2^{m} e^{\frac{a}{2} \rho \lambda^2 m(m-1)}.
\]
Fix $c > 4\sqrt{a\log2}$. Using (the exponential version of) Chebyshev's inequality, we deduce that for every $\lambda \in [0,1]$,
\[
  \mathbb P[ X' \geq \frac{c}{2} m \sqrt{(m-1) \rho}] \leq \mathbb E[e^{\lambda (X'-\frac{c}{2} m \sqrt{(m-1) \rho})}] \leq 2^m e^{ \frac{a}{2} \rho \lambda^2 m(m-1) - \frac{c}{2} m \sqrt{(m-1) \rho} \lambda}.
\]
Taking $\lambda = \frac{c}{2a \sqrt{(m-1)\rho}}$ (which is indeed in $[0,1]$ if $m \geq e^{\frac{c^2}{4a^2}}$, because we assumed that $\frac{\log m}{m} \leq \rho \leq 1$), we obtain
\[
  \mathbb P[ X' \geq \frac{c}{2} m \sqrt{(m-1) \rho}] \leq \mathbb E[e^{\lambda (X'-\frac{c}{2} m \sqrt{(m-1) \rho})}] \leq 2^m e^{ -\frac{c^2}{8a}m}.
\]
Since the random variables $X'$ and $X''$ are identically distributed, we have the same inequality for $X''$. Hence, taking into account that $X := \sum_i |d(i) - (m-1)\rho| \leq X'+X''$, we obtain
\[
  \mathbb P[ X \geq c m \sqrt{(m-1) \rho}] \leq 2^{m+1} e^{ -\frac{c^2}{8a}m} \leq e^{(2\log2-\frac{c^2}{8a})m}
\]
for every $m \geq e^{\frac{c^2}{4a^2}}$. The first inequality of the lemma follows.

The second inequality is a direct consequence of the first one. Indeed, by the triangle inequality, we have
\[
  \sum_i |d(i) - \overline{d}| \leq \frac{1}{m} \sum_{i,j} |d(i) - d(j)| \leq 2 \sum_i |d(i) - (m-1) \rho|,
\]
and therefore,
\begin{equation} \label{eq:relation_between_two_averages}
  \sum_i \frac{|d(i) - \overline{d}|}{m\overline{d}} \leq \frac{2(m-1)\rho}{ \overline d} \sum_i \frac{|d(i) - (m-1) \rho|}{m(m-1)\rho}.
\end{equation}
Let $A$ be the event where $\sum_i \frac{|d(i) - (m-1) \rho|}{m(m-1)\rho} \leq \frac{c_0}{\sqrt{(m-1)\rho}}$, so that by the first inequality (which we have just proved), we have $\mathbb P[A] \geq 1 - e^{-\frac{m}{c_0}}$. On $A$, we have
\[
  \overline{d} \geq (m-1)\rho - \frac{1}{m}\sum_i |d(i) - (m-1) \rho| \geq (m-1)\rho \left(1- \frac{c_0}{\sqrt{(m-1)\rho}}\right),
\]
which is greater than $\frac 1 2 (m-1) \rho$ for $m \geq m_0 = e^{2c_0+1}+1$. So by \eqref{eq:relation_between_two_averages}, we obtain
\[
  \sum_i \frac{|d(i) - \overline{d}|}{m\overline{d}} \leq \frac{4 c_0}{\sqrt{(m-1)\rho}}
\]
on $A$ if $m \geq m_0$. This proves the second inequality for $m \geq m_0$ with $c_0$ replaced by $4c_0$.
\end{proof}

We will also need that, above the connectivity threshold, Erd\H{o}s-R\'enyi graphs have a good two-sided spectral gap.
\begin{thm}[\cite{hoffmanetal}] \label{thm:hoffmanetal}
  Let $\eta>0$. There are constants $C'$ and a sequence $(u_m)$ of positive real numbers tending to $0$ such that for every $\rho \geq (1+\eta) \frac{\log m}{m}$, with probability $\geq 1-u_m$, an Erd\H{o}s-R\'enyi graph $\mathbb{G}(m,\rho)$ satisfies $\|A^0\| \leq \sqrt{\frac{C'}{m \rho}}$.
\end{thm}
\begin{proof}
This result was announded in \cite{hoffmanetal} (see \cite{MR2351691} for the statement when $\eta$ is sufficiently large). Formally, the result in \cite{hoffmanetal} deals with the giant component, but in the regime $\rho \geq \frac{(1+\eta)\log m}{m}$, which we are interested in, an Erd\H{o}s-R\'enyi graph $\mathbb{G}(m,\rho)$ is connected with high probability. The theorem also follows by combining Lemma \ref{lem:concentration_degree} and Theorem 3.2 in \cite{benaychbordenaveknowles}. Indeed the results in \cite{benaychbordenaveknowles} deal with the unnormalized adjacency matrix, which is unitarily conjugate to $D^{\frac 1 2} A D^{\frac 1 2}$, where $A$ is the Markov operator and $D$ the matrix with diagonal entries equal to the degrees.
\end{proof}

\subsection{Spectral gap and union of graphs}
We end this section by two results that indicate how a two-sided spectral gap behaves when one takes the union of two (non-random) graphs on the same vertex set.

For a weight $\omega:V \times V \to \mathbb{R}_+$, let $L^2(V,d_\omega)$ denote the space of square-integrable functions on $V$ equipped with the inner product given by
\begin{equation} \label{eq:innerproduct}
 \langle f,g \rangle = \sum_{s \in V} f(s)\overline{g(s)}d_\omega(s).
\end{equation}
(Note that this is different from Section \ref{sec:graphs}, in which we considered $L^2$-functions with respect to a stationary measure $\nu$.)

In the situation of Lemma \ref{lem:spectral_gap_under_perturbation1} and Proposition \ref{prp:spectral_gap_under_perturbation2}, define $\iota_1 \colon L^2(V,d_{\omega_1}+d_{\omega_2}) \to L^2(V,d_{\omega_1})$ and $\iota_2 \colon L^2(V,d_{\omega_1}+d_{\omega_2}) \to L^2(V,d_{\omega_2})$ by the formal identities $\iota_1 f=f$ and $\iota_2 f=f$. For all $f,g \in L^2(V,d_{\omega_1}+d_{\omega_2})$, we have
\begin{equation} \label{eq:markov}
  \langle A_{\omega_1+\omega_2} f, g\rangle = \langle A_{\omega_1} \iota_1f, \iota_1 g\rangle + \langle A_{\omega_2} \iota_2f, \iota_2 g\rangle,
\end{equation}
so that $A_{\omega_1+\omega_2} = \iota_1^* A_{\omega_1} \iota_1+\iota_2^* A_{\omega_2} \iota_2$.

The first lemma deals with the situation when one of the graphs is a small pertubation of the other. We could not locate this precise statement in the literature, but the argument is completely standard (see \cite[Lemma 4.5]{MR3106728}, \cite[Lemma 1.7]{MR3305311} or \cite[Lemma 9.2]{drutumackay}).
\begin{lem} \label{lem:spectral_gap_under_perturbation1}
  Let $V$ be a finite set, let $\omega_1,\omega_2: V \times V \to \R_+$ be two weight functions, and let $d_{\omega_1},d_{\omega_2}$ be the corresponding degree functions. As before, let $A_{\omega_i}^0$ denote the restriction of $A_{\omega_i}$ to $L^2_0(V,d_{\omega_i})$.
  
  If $d_{\omega_2} \leq \delta' d_{\omega_1}$ for some $\delta' > 0$, then
  \begin{equation} \label{eq:norm_in_L^0_under_perturbation}
    \left|\| A_{\omega_1+\omega_2}^0\| - \|A_{\omega_1}^0\|\right| \leq \delta'.
  \end{equation}
\end{lem}
\begin{proof}
Suppose that $d_{\omega_2} \leq \delta' d_{\omega_1}$. Using the standard min-max formulas for the eigenvalues of a self-adjoint matrix, and using that $\|\iota_1\|^2 = \sup_{s \in V} \frac{d_{\omega_1}(s)}{d_{\omega_1}(s)+d_{\omega_2}(s)}$ (and similar for $\|\iota_2\|^2$), we obtain
\[
  \left|\| A_{\omega_1+\omega_2}^0\| - \frac{1}{1+\delta'}\|A_{\omega_1}^0\|\right|\leq  \frac{\delta'}{1+\delta'},
\]
which is just a sharper version of \eqref{eq:norm_in_L^0_under_perturbation}.
\end{proof}

The second result deals with the situation where the degree sequence in each graph is well concentrated in $L^1$-average. When the degree sequence is uniformly ($L^\infty$) concentrated, such a result is well-known (see \cite{MR3106728,MR3305311,drutumackay}), but we will be in the situation where the degree sequence is not $L^\infty$-concentrated, and we need the following strong result. To our knowledge, this result is new, and it is the main result of this section.
\begin{prp}\label{prp:spectral_gap_under_perturbation2}
  Let $V$ be a finite set, let $\omega_1,\omega_2: V \times V \to \R_+$ be two weight functions, and let $d_{\omega_1},d_{\omega_2}$ be the corresponding degree functions. As before, let $A_{\omega_i}^0$ denote the restriction of $A_{\omega_i}$ to $L^2_0(V,d_{\omega_i})$.

  Also, let $D_i = \sum_{s \in V} d_{\omega_i}(s)$ be the total degree for $\omega_i$, and let
  \[
    \delta = \sum_{s \in V} \frac{\left|d_{\omega_1}(s)-\frac{D_1}{|V|}\right|}{D_1} + \frac{\left|d_{\omega_2}(s)-\frac{D_2}{|V|}\right|}{D_2}.
  \]
  Then
  \[ \| A_{\omega_1+\omega_2}^0\| \leq \delta+  (1-\delta) \max(\|A_{\omega_1}^0\|,\|A_{\omega_2}^0\|).\]
\end{prp}
\begin{proof}
Denote by $1=\lambda_1^{(i)}\geq \lambda_2^{(i)} \geq \dots \geq \lambda_n^{(i)}$ (with $n=|V|$) the eigenvalues of $A_{\omega_i}$, so that $\|A_{\omega_i}^0\| = \max(\lambda_2^{(i)},-\lambda_n^{(i)})$. We will prove that
\[
    \min(\lambda_n^{(1)},\lambda_n^{(2)})  \leq \langle A_{\omega_1+\omega_2} f,f\rangle \leq \delta + (1-\delta) \max(\lambda_2^{(1)},\lambda_2^{(2)})
\]
  for every $f \in L^2_0(V,d_{\omega_1}+d_{\omega_2})$ of norm one, which directly implies the proposition. The first inequality is obvious by \eqref{eq:innerproduct}. For the second one, write $\iota_1 f = f'_1+f_1''$ where $f_1'$ is the orthogonal projection of $f_1$ onto the constant functions in $L^2(V,d_{\omega_1})$. We have
\[
  \langle A_{\omega_1} \iota_1f, \iota_1 f\rangle \leq \|f_1'\|^2 + \lambda_2^{(1)}\| f_1''\|^2,\]
  and similarly for $\langle A_{\omega_2} \iota_2f, \iota_2 f\rangle$. By \eqref{eq:markov}, we obtain
\[
  \langle A_{\omega_1+\omega_2} f,f\rangle \leq \|f_1'\|^2+\|f'_2\|^2 + \lambda_2^{(1)}\| f_1''\|^2 + \lambda_2^{(2)}\| f_2''\|^2.
\]
Using that
\[
  1 =\|f\|^2 = \|\iota_1 f\|^2 + \|\iota_2 f\|^2 = \|f_1'\|^2+\|f'_2\|^2 + \| f_1''\|^2 + \| f_2''\|^2
\]
and writing $\alpha = \max(\lambda_2^{(1)},\lambda_2^{(2)})$, we obtain
\[
  \langle A_{\omega_1+\omega_2} f,f\rangle \leq \alpha + (\|f_1'\|^2+\|f'_2\|^2)(1-\alpha).
\]
So the lemma will be proved once we show that $\|f_1'\|^2+\|f'_2\|^2 \leq \delta$. Actually, we will show that
\begin{equation}\label{eq:norm_of_projection}
  \|f'_1\|^2 \leq \frac{D_2}{D_1+D_2} \delta
\end{equation}
and
\[
  \|f'_2\|^2 \leq \frac{D_1}{D_1+D_2} \delta.
\]
By symmetry, it suffices to show \eqref{eq:norm_of_projection}. By definition, $f'_1$ is the orthogonal projection onto the constant functions in $L^2(V,d_{\omega_1})$, so it is the constant function equal to $c_1=\frac{1}{D_1} \sum_{s \in V} d_{\omega_1}(s) f(s)$, which we write as the scalar product, in $L^2(V,d_{\omega_1}+d_{\omega_2})$, of $f$ with the function $s \mapsto \frac{d_{\omega_1}(s)}{D_1(d_{\omega_1}(s)+d_{\omega_2}(s))}$. Using the fact that $f$ is orthogonal to the constant functions (in $L^2(V,d_{\omega_1}+d_{\omega_2})$), this scalar product is equal to the scalar product of $f$ with
\[
  s \mapsto \frac{d_{\omega_1}(s)}{D_1(d_{\omega_1}(s)+d_{\omega_2}(s))} - \frac{1}{D_1+D_2} = \frac{d_{\omega_1}(s) D_2 - d_{\omega_2}(s) D_1}{D_1(D_1+D_2)(d_{\omega_1}(s)+d_{\omega_2}(s))}.
\]
By the Cauchy-Schwarz inequality in $L^2(V,d_{\omega_1}+d_{\omega_2})$, we obtain
\[
|c_1|^2 \leq \|f\|^2 \sum_{s \in V} \frac{|d_{\omega_1}(s) D_2 - d_{\omega_2}(s) D_1|^2}{D_1^2 (D_1+D_2)^2 (d_{\omega_1}(s)+d_{\omega_2}(s))}.
\]
Using the following estimate:
\begin{eqnarray*}
  |d_{\omega_1}(s) D_2 - d_{\omega_2}(s) D_1|^2 &\leq& (d_{\omega_1}(s) D_2 + d_{\omega_2}(s) D_1) |d_{\omega_1}(s) D_2 - d_{\omega_2}(s) D_1|\\ &\leq& (d_{\omega_1}(s)+d_{\omega_2}(s))(D_1+D_2) |d_{\omega_1} (s) D_2 - d_{\omega_2}(s) D_1|,
\end{eqnarray*}
we obtain
\[
  |c_1|^2 \leq \sum_{s \in V} \frac{|d_{\omega_1}(s) D_2 - d_{\omega_2}(s) D_1|}{D_1^2 (D_1+D_2)}.
\]
By the triangle inequality, we have
\[
  |d_{\omega_1}(s) D_2 - d_{\omega_2}(s) D_1| \leq \left|d_{\omega_1}(s)-\frac{D_1}{n}\right| D_2 + \left|d_{\omega_2}(s)-\frac{D_2}{n}\right| D_1,
\]
where $n=|V|$ (as before). Finally, summing over $s$ yields
\[
  |c_1|^2 \leq \frac{D_2}{D_1 (D_1+D_2)} \delta.
\]
Since
\[
  \|f'_1\|^2 = \sum_{s \in V} d_{\omega_1}(s) |c_1|^2 = D_1 |c_1|^2,
\]
we obtain \eqref{eq:norm_of_projection}, which concludes the proof of the lemma.
\end{proof}

\section{Fixed point properties for random groups} \label{sec:randomgroups}
In this section, we apply our spectral criterion (Theorem \ref{thm:criterionuc}) to random groups in the triangular model. In particular, we will prove Theorem \ref{thm:maintheoremrandomgroups} and Corollary \ref{cor:randomgroupslp}.

Let $S=\{s_1,\ldots,s_m\}$. Roughly speaking, a random group generated by $S$ is a group given by a representation $\langle S \vert R \rangle$, where $R$ is a set of relators, i.e.~words in $S \cup S^{-1}$, that are chosen randomly with respect to some probability measure. In what follows, we only consider relators that are cyclically reduced, i.e.~relators of the form $r=s_1...s_l$ with $s_i \neq s_{i+1}^{-1}$ for $i \in \{1,\ldots,l-1\}$ and $s_l \neq s_1^{-1}$. The number of cyclically reduced words of length $3$ is $(2m-1)^3+1$.

The triangular density model $\mathcal{M}(m,d)$ was introduced by \.{Z}uk in \cite{MR1995802}. For a fixed density $d \in (0,1)$, a group in the model $\cM(m,d)$ is a group $\Gamma = \langle S \vert R \rangle$, where $|S|=m$ and $R$ is a set of relators, chosen uniformly among all subsets of cardinality $(2m-1)^{3d}$ (rounded to the nearest integer) of the set of cyclically reduced relators of length $3$. A property $P$ for groups is said to hold with overwhelming probability (w.o.p.) in the triangular density model $\cM(m,d)$ if
\[
  \lim_{m \to \infty} \mathbb{P}(\Gamma \textrm{ in } \cM(m,d) \textrm{ has } P)=1.
\]
In the proofs below, we will use the following -- to our purposes more convenient -- version of the triangular model, which was also used by Dru\c{t}u and Mackay \cite{drutumackay}.
\begin{dfn}
  For natural numbers $m$ and $N$ with $N \leq (2m-1)^3+1$, a group in the triangular model $\cM'(m,N)$ is a group $\Gamma = \langle S \vert R \rangle$, where $|S|=m$ and $R$ is a random set of relators, chosen uniformly among all subsets of cardinality $N$ of the set of cyclically reduced relators of length $3$.
\end{dfn}
A closely related model is the binomial triangular model $\Gamma(m,\rho)$ (see \cite{MR3305311}).
\begin{dfn}
  For $\rho \in (0,1)$, a group in the binomial triangular model $\Gamma(m,\rho)$ is a group $\Gamma = \langle S \vert R \rangle$, where $|S|=m$ and $R$ is a random set of relators, where each cyclically reduced relator of length $3$ is chosen independently with probability $\rho$.
\end{dfn}
A property $P$ for groups is said to hold with overwhelming probability (w.o.p.) in the binomial triangular model $\Gamma(m,\rho)$ if
\[
  \lim_{m \to \infty} \mathbb{P}(\Gamma \textrm{ in } \Gamma(m,\rho) \textrm{ has } P)=1.
\]
In the model $\cM'(m,N)$, each relator appears with probability $\frac{N}{(2m-1)^3+1}$, so the models $\cM'(m,N)$ and $\Gamma(m,\rho)$ are closely related when $N$ and $\rho$ are related through $\rho \sim \frac{N}{(2m-1)^{3}}$.

In order to prove structural properties in the binomial triangular model, one can typically rely on results on Erd\H{o}s-R\'enyi graphs. In particular, we will heavily rely on the results that we proved in Section \ref{sec:ergraphs}.\\

We prove the following result for fixed point properties of actions of random groups in the binomial triangular model, and then we explain how it implies Theorem \ref{thm:maintheoremrandomgroups}.
\begin{thm} \label{thm:binom_triang_model}
Let $\eta>0$. There is a constant $B > 0$ and a sequence $(u_m)$ of positive real numbers tending to $0$ such that the following holds: Let $m \in \N$ and $\rho \in (0,1)$. If $\rho \geq \frac{(1+\eta) \log m}{8 m^2}$, then, with probability $\geq 1-u_m$, a group in $\Gamma(m,\rho)$ has property (F$_X$) for every uniformly curved Banach space satisfying $\rho m^2 \varepsilon(X)^2 \geq B$. In particular, the latter is the case when $\varepsilon(X) \geq \sqrt{\frac{8B}{\log m}}$.
\end{thm}
By the relationship between the models $\Gamma(m,\rho)$ and $\cM'(m,N)$ (see \cite[Section 10]{drutumackay}) this theorem immediately implies the following result.
\begin{cor} \label{cor:triangularmodel}
Let $\eta>0$. There is a constant $B'$ and a sequence $(u_m)$ of positive real numbers tending to $0$ such that the following holds: Let $m \in \N$ and $N \in [(1+\eta) m \log m,(2m-1)^3+1]$. With probability $\geq 1-u_m$, a group in $\cM'(m,N)$ has property (F$_X$) for every uniformly curved Banach space satisfying $\varepsilon(X) \geq \sqrt{\frac{B'm}{N}}$. In particular, the latter is the case when $\varepsilon(X) \geq \sqrt{\frac{B'}{\log m}}$.
\end{cor}
\begin{proof}
  If $\rho \sim \frac{N}{(2m-1)^{3}}$, then there exists an $\eta'>0$ such that the condition $N \geq (1+\eta) m \log m$ implies $\rho \geq \frac{(1+\eta') \log m}{8 m^2}$. 
\end{proof}
We can now prove Theorem \ref{thm:maintheoremrandomgroups}.
\begin{proof}[Proof of Theorem \ref{thm:maintheoremrandomgroups}]
The proof is a direct application of Corollary \ref{cor:triangularmodel}. Indeed, let $d$ satisfy \eqref{eq:condition_on_d} for some $\eta \in (0,2)$, i.e.
\[
  d \geq \frac 1 3 + \frac{\log \log m - \log(2-\eta)}{3 \log m}.
\]
It follows that there is an $\eta'>0$ such that for all $m$ sufficiently large, we have
\[
(2m-1)^{3d} \geq (1+\eta') m \log m,
\]
which implies the result.
\end{proof}
We deduce Corollary \ref{cor:randomgroupslp}.
\begin{proof}[Proof of Corollary \ref{cor:randomgroupslp}]
This is a direct application of Theorem \ref{thm:maintheoremrandomgroups}, since for every $\eta>0$, for all $m$ sufficiently large, we have the following:
  \begin{itemize}
  \item The condition $p \leq \sqrt{\frac{(3d-1) \log m}{\eta + \log 2}}$ implies $2 p^{-\frac p 2} 2^{-\frac{p^2}{2}} \geq \sqrt{\frac{C m }{(2m-1)^{3d}}}$.   
  \item For every real $K \geq 1$, the condition $p \leq \sqrt{\frac{(3d-1) \log m}{\eta + \log (2K)}} - \frac 1 2$ implies
    \[( 2 p^{-\frac p 2} 2^{-\frac{p^2}{2}} )/K^{\frac{p(p+1)}{2}} \geq \sqrt{\frac{C m }{(2m-1)^{3d}}}.\]
  \end{itemize}
  In particular using the explicit estimates for $\varepsilon(X)$ in Remark \ref{rem:explicit_computation_of_epsilonX}, we obtain the following:
  \begin{itemize}
  \item The condition $p \leq \sqrt{\frac{(3d-1) \log m}{\eta + \log 2}}$ implies $ \varepsilon(L^p) \geq \sqrt{\frac{C m }{(2m-1)^{3d}}}$.   
  \item For every $\alpha \geq 1$ and every Banach space $X$ which is $\alpha$-isomorphic to a subquotient to a $\frac 2 p$-Hilbertian space, the condition $p \leq \sqrt{\frac{(3d-1) \log m}{\eta + \log (2\alpha)}} - \frac 1 2$ implies $\varepsilon(X) \geq \sqrt{\frac{C m }{(2m-1)^{3d}}}$.
  \end{itemize}
\end{proof}
Let us now prove Theorem \ref{thm:binom_triang_model}. For every triangular presentation $\Gamma = \langle S \vert R \rangle$, i.e.~a presentation in which every relator has length $3$, the associated Cayley complex is a simplicial complex, in which the links are all isomorphic to the graph $L:=L(S,R)$ decribed as follows. The graph $L$ has vertex set $S \cup S^{-1}$ and edges $\{s_x^{-1},s_y\}$, $\{s_y^{-1},s_z\}$ and $\{s_z^{-1},s_x\}$ whenever $s_xs_ys_z \in R$. Note that the edges come in three types, corresponding to the order in which the generators $s_x$, $s_y$ and $s_z$ occur in the relation $s_xs_ys_z$. This order yields a decomposition of $L$ into three graphs $L^1$, $L^2$ and $L^3$, where $L^i$ has the same vertices as $L$, but only the edges corresponding to the appropriate place in the relation. Since such links often have multiple edges, we will work in the setting of weighted graphs.

When $\rho \geq m^{-1.42}$,  $\Gamma(m,\rho)$ is finite with overwhelming probability \cite{MR1995802}, and in particular Theorem \ref{thm:binom_triang_model} holds for trivial reasons. For $\rho \leq m^{-1.42}$ Theorem \ref{thm:binom_triang_model} is an immediate consequence of Theorem \ref{thm:criterionuc} and of the following result.
\begin{prp} \label{prp:links}
Let $\eta>0$. There is a constant $C$ and a sequence $(u_m)$ tending to $0$ such that the following holds: Let $m \in \N$ and $\rho \in (0,m^{-1.42})$. If $\rho \geq \frac{(1+\eta) \log m}{8m^2}$, then the link $L:=L(S,R)$ of a group presentation in $\Gamma(m,\rho)$ satisfies
  \[
    \|A_L\|_{L^2_0(L,\nu)} \leq \sqrt{\frac{C}{\rho m^2}}
  \]
with probability at least $1-u_m$.
\end{prp}
The proof of Proposition \ref{prp:links} partly follows the strategy of \cite{MR3305311}. However, we need some adaptations because we want to deal with a more general range of $\rho$ where the degree sequence in the links is not well concentrated (see Lemma \ref{lem:spectral_gap_under_perturbation1}), and we aim for a stronger conclusion.
\begin{proof}
It is explained in \cite{MR3305311} that each $L^i$ (as defined above) is essentially an Erd\H{o}s-R\'enyi graph $\mathbb{G}(2m,\rho')$ with $\rho'=1-(1-\rho)^{4m-2}$.

Let us recall how the reduction to the Erd\H{o}s-R\'enyi model works. It follows from the definition that each $L^i$ is a graph with vertex set $S \cup S^{-1}$, in which the weights $\omega_i(s,t)$ (for a set $\{s,t\}$ of vertices of cardinality $2$) are independent and with binomial distribution $B(4m-4,\rho)$ if $s \neq t^{-1}$ and $B(4m-2,\rho)$ otherwise. Of course, the weights $\omega_1$, $\omega_2$ and $\omega_3$ are not independent.

Let $L_i'$ be the graph obtained from $L^i$ by adding, independently, to every edge of the form $\{s,t\}$ with $s \neq t^{-1}$ a binomial $B(2,\rho)$, so that the weights $\omega'_i(s,t)$ are independent and identically distributed binomial random variables $B(4m-2,\rho)$. The total number of edges that are added in this way to each vertex $s \in S\cup S^{-1}$ is a sum of $2m-2$ independent binomial $B(2,\rho)$ variables. Hence, it is a $B(4m-4,\rho)$ random variable. Since the probability that a binomial $B(N,\rho)$ is strictly greater than $2$ is less than $N^3 \rho^3$, we have that with probability $\geq 1-2m (4m-4)^3 \rho^3 \geq 1-O(m^{-0.2})$, to each $s \in S$ we have added at most $2$ edges.

Let $L''_i$ be the graph obtained from $L_i'$ by replacing all multiple edges by simple edges, i.e.~by replacing $\omega'_i$ by $\omega''_i(s,t)=\min(\omega'_i(s,t),1)$. The graph $L''_i$ is an Erd\H{o}s-R\'enyi graph $\mathbb{G}(2m,\rho')$, where $\rho'=1-(1-\rho)^{4m-2}$, and $\rho' \sim 4m \rho \geq (1+\frac{\eta}{2}) \frac{\log m}{2m}$ for $m$ sufficiently large. We make the following claims:
\begin{enumerate}[(i)]
 \item\label{item:union_of_ErdosRenyi} $L'':=L''_1\cup L''_2 \cup L''_3$ satisfies the conclusion of the proposition for a certain constant $C' > 0$.
 \item\label{item:degree_of_ErdosRenyi} There is a constant $C''>0$ such that with high probability, we have $d_{\omega''}(s) \in [C''^{-1}m^2 \rho,C''m^2 \rho]$ for every $s \in S \cup S^{-1}$.
 \item\label{item:degree_added} There exists a constant $C''' > 0$ such that with probability $1-o(1)$, we have $\sup_s d_{\omega'-\omega''}(s) \leq C''' \sqrt{m^2\rho}$.
\end{enumerate}
These claims imply the proposition, possibly by replacing the constants. Indeed, by (\ref{item:degree_of_ErdosRenyi}) and (\ref{item:degree_added}) we see that, with high probability, $L''$ is obtained from $L$ by adding and then removing edges in a way that changes the degree of every vertex by a factor in $\left[1-O\left(\frac{1}{\sqrt{m^2 \rho}}\right),1+O\left(\frac{1}{\sqrt{m^2 \rho}}\right)\right]$. Lemma \ref{lem:spectral_gap_under_perturbation1} therefore implies that, with high probability,
\[
  \|A_L^0\| \leq \|A_{L''}^0\|+ \frac{C'}{\sqrt{m^2 \rho}}.
\]
By \eqref{item:union_of_ErdosRenyi} we obtain the conclusion.

Let us prove the claims.

(\ref{item:union_of_ErdosRenyi}): It follows from Theorem \ref{thm:hoffmanetal} that there is a constant $C' > 0$ such that if $\rho' \geq \frac{(1+\eta) \log (N)}{N}$, then a graph in $\mathbb G(N,\rho')$ satisfies $\|A^0\| \leq \sqrt{\frac{8C'}{N\rho'}}$ with high probability. Here we have $N=2m$ and $\rho'= 1-(1-\rho)^{4m-4} \simeq 4m \rho \geq \frac{(1+\eta)\log m}{2m}$, so we get that, with high probability, $\|A_{L''_i}^0\|\leq \sqrt{\frac{C'}{m^2 \rho}}$. One concludes by Lemma \ref{lem:average_concentration} and Proposition \ref{prp:spectral_gap_under_perturbation2}.

(\ref{item:degree_of_ErdosRenyi}): This is Lemma \ref{lem:concentration_degree}.

(\ref{item:degree_added}): Denote by $\varepsilon_i(s,t)$ the random variable equal to $1$ if $(\omega'_i-\omega''_i)(s,t)>0$ and equal to $0$ otherwise. These are independent and identically distributed Bernoulli variables with expectation $O(m^2 \rho^2)$. Since the probability that $\omega'_i(s,t)$ is strictly greater than $4$ is $O((\rho m)^5) = O(m^{-2.1})$, we obtain that with probability $1-o(1)$, we have $\max_{s,t} \omega'_i(s,t) \leq 4$, i.e. $\omega'_i-\omega''_i \leq 4 \varepsilon_i$. So it is enough to prove that
\begin{equation}\label{eq:concentation_binomial} \sup_s \sum_t \varepsilon_i(s,t) \leq \frac{\sqrt{m^2\rho}}{4}\end{equation}
with probability $1-o(1)$. But for each $s$, $\sum_t \varepsilon_i(s,t)$ is a binomial random variable with expectation $O(m^3 \rho^2)$. Moreover by the assumption that $\rho \leq m^{-1.42}$,  $m^3 \rho^2= o(\sqrt{m^2 \rho})$, so \eqref{eq:concentation_binomial} follows from the standard concentration phenomenon for binomial random variables.
\end{proof}

\end{document}